\documentclass[a4paper,11pt]{article}
\usepackage{geometry}
\usepackage{url}
\usepackage[colorlinks]{hyperref}
\usepackage{amssymb}
\usepackage{amsfonts}
\usepackage{amsmath}
\usepackage{amsthm}
\usepackage{amscd}
\usepackage{txfonts}
\usepackage{latexsym}
\usepackage{epsfig}
\usepackage{graphicx}
\usepackage[english]{babel}
\usepackage[utf8]{inputenc}
\usepackage{fancyvrb}
\usepackage{float}
\usepackage[all]{xy}

\newcommand{\mathsym}[1]{{}}

\newcommand{\eps}{\varepsilon}
\renewcommand{\phi}{\varphi}

\newcommand{\Ho}{\mathcal{H}}

\newcommand{\A}{\mathcal{A}}
\newcommand{\bN}{\mathbb{N}}

\newcommand{\C}{\mathcal{C}}

\newcommand{\B}{\mathcal{B}}
\newcommand{\M}{\mathcal{M}}

\newcommand{\N}{\mathbb{N}}
\newcommand{\fF}{\mathfrak{F}}
\newcommand{\fU}{\mathfrak{U}}

\newcommand{\Disk}{\mathbb{D}}

\newcommand{\BH}{\mathcal{B}(\mathcal{H})}
\newcommand{\KH}{\mathcal{K}(\mathcal{H})}
\renewcommand{\ker}{\operatorname{ker}}
\newcommand{\Rang}{\operatorname{Ran}}

\newcommand{\support}{\operatorname{supp}}

\newcommand{\aid}{e} 
\newcommand{\Inv}{\operatorname{Inv}}
\newcommand{\TopInv}
{\operatorname{TopInv}}
\newcommand{\AppInv}
{\operatorname{AppInv}}
\newcommand{\TopQInv}{\operatorname{TopQInv}}
\newcommand{\Lone}{L^1}

\newcommand{\Lp}{L^p}
\newcommand{\Characters}[1]{\mathcal{M}_{#1}}
\newcommand{\dif}{\mathrm{d}}
\newcommand{\eqdef}{:=}
\newcommand{\DiskAlgebra}{\mathcal{A}(\mathbb{D})}
\newcommand{\SmallDiskAlgebra}{\mathcal{A}_0(\mathbb{D})}

\newcommand{\CompactOperators}{\mathcal{K}}
\newcommand{\monomial}{\chi_1}
\newcommand{\Hull}{\mathcal{H}}
\newcommand{\linspan}{\operatorname{span}}
\newcommand{\cF}{\mathcal{F}}

\newcommand{\hspn}{\hspace{0pt}}

\newtheorem{theorem}{Theorem}[section]
\newtheorem{corollary}[theorem]{Corollary}
\newtheorem{lemma}[theorem]{Lemma}
\newtheorem{proposition}[theorem]{Proposition}

\newtheorem{question}[theorem]{Question}
\theoremstyle{definition}
\newtheorem{example}{Example}[section]
\newtheorem{definition}[theorem]{Definition}
\newtheorem{remark}[theorem]{Remark}

\usepackage[colorinlistoftodos]{todonotes}

\definecolor{darkgreen}{rgb}{0,0.7,0}

\numberwithin{equation}{section}
\renewenvironment{proof}[1][\bf{Proof}]{\noindent\textsc{#1.} }{\quad\hfill$\blacksquare$ }

\usepackage{enumitem}

\title{Approximately invertible elements\\in non-unital normed algebras}
\author{\texorpdfstring{Kevin Esmeral,
		Hans G. Feichtinger,
		Ondrej Hutn\'{i}k,
		Egor A. Maximenko}%
	{Kevin Esmeral,
		Hans G. Feichtinger,
		Ondrej Hutn\'{i}k,
		Egor A. Maximenko}}
\date{}
\begin{document}
	\maketitle

	\begin{abstract}
		We introduce a concept of
		\emph{approximately invertible elements} in non-unital normed algebras which is, on one side, a natural generalization of invertibility when having approximate identities at hand, and,
		on the other side, it is a direct extension of topological invertibility to non-unital algebras. Basic observations relate approximate invertibility with concepts of topological divisors of zero and density of (modular) ideals. We exemplify approximate invertibility in the group algebra, Wiener algebras, and operator ideals.
		For Wiener algebras with approximate identities (in particular, for the Fourier image of the convolution algebra), the approximate invertibility of an algebra element is equivalent to the property that it does not vanish.  
		We also study approximate invertibility and its deeper connection with the Gelfand and representation theory in non-unital abelian Banach algebras as well as abelian and non-abelian C*-algebras.

		\smallskip\noindent
		MSC (2020): 46H10, 46L05, 43A20.
		
		
		\smallskip\noindent
		\textbf{Keywords}: approximate identity, approximate invertibility, maximal modular ideal, principal ideal, non-unital Banach algebra, non-unital topological algebra.
	\end{abstract}
	
	\tableofcontents
	
	\section{Introduction}
	
	Invertibility is one of the central concepts in the study of unital Banach (or, more generally, topological) algebras. However, this concept is closely related to the existence of an identity in the algebra. Every non-unital Banach algebra $\A$ may be embedded into a unital algebra $\A_1$; such unitization is very useful for some goals,
	but it is not very appropriate for others. In particular, every element of the original algebra $\A$
	is not invertible in the unital algebra $\A_1$.
	
	An important tool in non-unital Banach algebras is the concept of \emph{approximate identity}, which serves as a good substitute for an identity. This concept goes back to the earliest studies of the group algebra $\Lone(G)$ and in this case the approximate identities have been systematically investigated by Weil in his book~\cite{Weil}. However, certain concrete approximate identities were considered long before the abstract definition was introduced. Approximate identities were apparently first considered explicitly by Segal in~\cite{Segal}, who constructed an approximate identity bounded by one for any norm-closed self adjoint subalgebra of the algebra of bounded linear operators on a Hilbert space. Then Dixmier used the approximate identity as the main tool in his book~\cite{Dixmier} to carry through all the basic theory of C*-algebras. Various results on left, right and two-sided approximate identities are described in Dixon's papers~\cite{Dixon1} and~\cite{Dixon2}. From that time many results known for C*-algebras, or for unital algebras, have been extended to algebras with approximate identity, see many books on the topic, e.g.~\cite{Doran-Wichmann}, \cite{Kaniuth}, \cite{Larsen}, and \cite{Palmer}.

	The topological divisors of zero as well as approximate (or, topological) identities are two instances suggesting the topologization of algebraic concepts. In this paper we propose a concept of \emph{approximate invertible elements} as an attempt to provide a ``topologization of invertibility'' in non-unital algebras. The main object of study of this paper reads as follows:

	\begin{definition}\label{def:ApproInv}
		An element $x$ of a topological algebra $\A$ is said to be \emph{approximately right invertible}
		if there is a net $(r_{j})_{j\in J}$ in $\A$ such that
		$(xr_{j})_{j\in J}$ is an approximate identity in $\A$. Similarly, an element $x\in\A$ is \emph{approximately left invertible}
		if there is a net $(l_{j})_{j\in J}$ in $\A$ such that
		$(l_{j}x)_{j\in J}$ is an approximate identity in $\A$.
	\end{definition}
	
	To our best knowledge, the suggested concept of approximate invertibility is not included in any available literature we have seen although it seems very natural in the context of algebras with approximate identities. A closely related notion in unital topological algebras is given by Thatte and Bhatt~\cite{ThatteBhatt} as a topological invertibility. This concept was further developed by Akkar et al.~\cite{Akkar-Beddaa-Oudadess}. In fact, both concepts coincide in unital topological algebras and, moreover, they collapse to invertibility in unital Banach algebras. However, classical invertibility and topological invertibility make no sense in non-unital algebras, thus approximate invertibility serves as an extension of these concepts to non-unital algebras. In the literature one can find yet another related concept -- a topological quasi-invertibility, see~\cite{Najmi, Zohri-Jabbari-2010}.
	Already in 2001 M.~Abel~\cite{Abel} provided a characterization of topological algebras in which the set of topologically quasi-invertible elements coincides with the set of quasi-invertible elements. 
	
	A similar terminology (approximately invertible maps, approximate inverses, or even approximate invertibility) is used in other contexts, and in a different sense, see for example~\cite{Boxer,Harte,Kromer,Schuster,Zames}. Moreover, in operator-theoretic community there is a notion of approximately invertible operators (by a sequence of approximate operators) studied e.g. in the book~\cite{HagenRochSilbermann}. Note that few authors use sometimes the term ``approximate invertible operators'' instead of ``Fredholm operators''.
	
	Regarding the operator theory, our motivation to introduce the approximate invertibility concept is related to projects dealing with density of the range of certain convolution operator arising in the study of Toeplitz and Toeplitz-type operators acting on various function spaces (usually, the weighted Bergman spaces over the upper half-plane, or the unit disk in the complex plane, \cite{EM,EV,EMV,HHM,HMV}, 
	as well as wavelet function spaces on the affine group \cite{H}). In these cases approximate inverses for some particular convolutions have been constructed. In particular, the main step in recent papers \cite{EM,HHM,HMV,HMM}   
	was to construct an appropriate Dirac net, and using this net to show a density result of a function algebra under consideration in certain C*-algebra (e.g., the algebra SO($\mathbb{N}$),  
	or the algebra VSO$(\mathbb{R}_+)$). 
	An idea of Wiener deconvolution technique on the real line has already been elaborated in~\cite{HMM}. 
	Immediately we have observed that the used techniques may be generalized to non-unital normed algebras leading to the concept of approximate invertibility as introduced in the paper.
	We generalize this idea in Section~\ref{sec:applications}.
	
	Although we will work mainly with normed algebras, approximate invertibility notion and many results of this paper can be easily generalized to topological algebras as well.

	The paper is organized as follows.
	In Section \ref{Sec:intro-app-inv} the notion of approximately invertible elements is introduced, exemplified and the relations with convergence, topological divisors of zero and ideals are studied in an elementary way. A detailed study of approximate invertibility in some classes of algebras such as
	Banach algebras, C*-algebras and involutive algebras is given in Section \ref{sec:AppInv in algebras}. In each case we aim to provide a characterization of approximately invertible elements by means of Gelfand transform, modular ideals, or non-degenerate representations, see Theorem~\ref{Appinvl-modi-C*} and Proposition~\ref{prop:criterion_ainv_in_csa}. 
	 
	Section~\ref{Sec:Examp} brings several interesting examples of algebras with or without approximately invertible elements. Particular examples have served us as a motivation for investigating approximate invertibility in some classes of algebras studied in Section \ref{sec:AppInv in algebras}. The most important (from the viewpoint of applications in the study of Toeplitz operator algebras and other parts of time-frequency analysis) is a Wiener algebra possessing an approximate identity where an element of this algebra is approximately invertible
	if and only if it does not vanish, see Theorem~\ref{thm:Wiener_ainv-Wie-alg}. 
	Finally, we provide a necessary and sufficient condition for the left and right approximate invertibility in operator ideals (including the C*-algebra of compact operators acting on an infinite-dimensional separable Hilbert space), see Theorem~\ref{prop:ApprInvCompact-opi}. An application to the density in Banach modules is given in Section~\ref{sec:applications}. As a by-product we get a result about density of the image of the convolution operator. In the last Section \ref{Sec:rem-open-pro} we present some questions, remarks and open problems related to approximately invertible elements for further investigation.
	
	\section{Invertibility in non-unital algebras}\label{Sec:intro-app-inv}

	If the invertibility cannot be used and we still wish to share the good properties of approximate identities, a new concept of \emph{approximately invertible} elements in non-unital normed algebras can be used. In this section we exemplify this concept, relate it with existing ones in the literature and study several elementary properties related to convergence, topological divisors of zero and ideals.

	\subsection{Approximate identities in action}
	
	First we summarize some basic properties of approximate identities in non-unital normed algebras. For details see   \cite{Conway,Dixon1,Dixon2,Dixon3,Doran-Wichmann,Feichtinger,Larsen,Murphy,Palmer,Rudin,Zelazko}.
	
	\begin{definition}\label{DefinitionAppId}
		An \emph{approximate identity}
		in a normed algebra $\A$ is a net $(\aid_{j})_{j\in J}$ in $\A$
		such that for every $x$ in $\A$ it holds
		$\displaystyle
		\lim_{j\in J}\aid_{j}x=\lim_{j\in J}x\aid_{j}=x.
		$
	\end{definition}
	In fact, in this definition it is sufficient to consider only non-zero elements $x\in \A$. In a similar way left and right approximate identity is defined. The existence of an approximate identity in a dense subset of algebra is immediate. Namely, if $\A$ is a normed algebra with an approximate identity
	and $S$ be a dense subset of $\A$, then $\A$ has an approximate identity with values in $S$, see \cite[Lemma 1.4]{Doran-Wichmann}.

	An approximately unital algebra shares some of the properties of a unital algebra.
	Obviously, if $\A$ is a unital algebra with unit $e$ and $J$ is an arbitrary directed set, then we can define an approximate identity $(\aid_{j})_{j\in J}$ in $\A$ easily by the rule $\aid_j=e$ for all $j\in J$. Also, from Definition~\ref{DefinitionAppId} it follows that if an approximate identity is a divergent net, then the normed algebra is non-unital. The following (in some sense reverse) observation is immediate.
	
	\begin{proposition}\label{prop:convergentAI}
		Let $(\aid_j)_{j\in J}$ be an approximate identity in $\A$ having a limit $u\in\A$. Then $u$ is the unit in $\A$.
	\end{proposition}

	On the one hand, unbounded approximate identities may look useless, see e.g.~\cite{Dixon3} for some pathological examples in incomplete normed algebras. On the other hand, they may be particularly useful in other contexts, e.g. in  so-called Segal algebras.

	For more details we refer the interested reader to \cite{ReiterStegeman}. Therefore we distinguish between \emph{norm bounded} and \emph{operator norm bounded} approximate identities in normed algebras.
	
	\begin{definition}\label{Defin:normb-opb-app-i}
		Let $(\A,\|\cdot\|)$ be a normed algebra and $(\aid_{j})_{j\in J}$ be a left approximate identity in $\A$. Then
		\begin{itemize}
			\item [\rm{(a)}] $(\aid_{j})_{j\in J}$ is said to be \emph{norm bounded} if there is a finite constant $M>0$ such that
			$\|\aid_{j}\|\leq M$ for every $j\in J$.
			\item [\rm{(b)}] $(\aid_{j})_{j\in J}$ is said to be \emph{operator norm bounded} if the net $(T_{j})_{j\in J}$ of left multiplication operators $T_{j}x=e_{j}x$ is such that
			$$\sup_{j\in J}\|T_{j}\|_{\textrm{op}}<+\infty,\,\,\, \text{where}\, \|\cdot\|_{\textrm{op}}\,\,\text{is the operator norm}.$$
		\end{itemize}
	\end{definition}

	The assertion of the following proposition permits us to extend operator norm bounded approximately identities in separable normed algebras to its completion. In addition,  it is used later in the proof of Theorem~\ref{thm:dense_ideals_implies_ainv} and in the section about convolution algebras.

	\begin{proposition}\label{prop:compl-opnb-algopnb}
		Let $(\A,\|\cdot\|_{\A})$ be a separable normed algebra and $(\tilde{\A},\|\cdot\|_{\tilde{\A}})$ be its completion. Then $(\tilde{\A},\|\cdot\|_{\tilde{\A}})$ has an (operator) norm bounded approximate identity $(\tilde{e}_{n})_{n\in \N}$ if and only if $(\A,\|\cdot\|_{\A})$ has an (operator) norm bounded  approximate identity.
	\end{proposition}
	\begin{proof} 
		The proof for bounded approximate identities may be found in~\cite[\textsection 6]{Reiter71}. Next, we give a proof for operator bounded approximate identities.
		
		Suppose that $(\tilde{\A},\|\cdot\|_{\tilde{\A}})$ has operator norm bounded approximate identity $(\tilde{e}_{n})_{n\in \N}$ and let us denote by $\tilde{T}_{n}$ the linear operators given by $\tilde{T}_{n}x=e_{n}x$, $x\in\tilde{\A}$, where  $M=\sup_{n\in\N}\|\tilde{T}_{n}\|_{\textrm{op}}<\infty$ by Definition \ref{Defin:normb-opb-app-i}. Then  by \cite[Lemma 1.4]{Doran-Wichmann} there exists a sequence $(e_{n})_{n\in\N}$ of elements in $\A$ such that $(e_{n})_{n\in\N}$ is an approximate identity of $\tilde{\A}$ (and of course in $\A$) with   $\|e_{n}-\tilde{e}_{n}\|_{\tilde{\A}}<\frac{1}{n}$. Denote by $T_{n}$ the linear operator acting on $\tilde{\A}$ given by $T_{n}x=e_{n}x$. Therefore, for every $x\in\A$
		\begin{align*}
			\|T_{n}x\|_{\A}&\leq \|T_{n}x-\tilde{T}_{n}x\|_{\tilde{\A}}+\|\tilde{T}_{n}x\|_{\tilde{\A}}\leq \|e_{n}-\tilde{e}_{n}\|_{\tilde{\A}}\,\|x\|_{\A}+M\,\|x\|_{\A}\\
			&\leq \dfrac{\|x\|}{n}+M\, \|x\|_{\A}\leq \left(1+\sup_{n\in\N}\|\tilde{T}_{n}\|_{\textrm{op}}\right)\,\|x\|_{\A}.
		\end{align*}
		Consequently, $\|T_{n}\|_{\textrm{op}}\leq M+1$ for every $n\in\N$ and  hence we have   $\sup_{n\in\N}\|T_{n}\|_{\textrm{op}}<+\infty$. i.e., $(e_{n})_{n\in\N}$ is an operator norm bounded approximate identity in $\A$.
		
		Conversely, it is an easy exercise left to the reader to extend an operator norm bounded approximate identity for $(\A,\|\cdot\|_{\A})$ to all of $(\tilde{\A},\|\cdot\|_{\tilde{\A}})$, observing that it is still an approximate identity there, and of course still bounded in the operator norm.
	\end{proof}
	
	\begin{remark}
		The proof of necessity of $(\A,\|\cdot\|_{\A})$ having an operator norm bounded approximate identity in Proposition \ref{prop:compl-opnb-algopnb} may be done using nets instead of sequences.
	\end{remark}
	
	\subsection{Approximate invertibility and related concepts}

	In a unital algebra $\A$ (with $e$ being the unit in $\A$) the invertibility of an element $x\in\A$ may be described using the net $y_j=x^{-1}$ for each $j\in J$ from a directed net $J$, such that for each $z\in\A$ it holds $$\lim_{j\in\,J}xy_jz = \lim_{j\in\,J}zxy_j = \lim_{j\in\,J}y_jxz = \lim_{j\in\,J}zy_jx = z.$$ It means that the nets $(xy_j)_{j\in\,J}$ and $(y_jx)_{j\in\,J}$ are approximate identities in $\A$. A natural generalization of this observation for the case of non-unital normed algebras is given in Definition~\ref{def:ApproInv} being the main object of our study.
	
	If the corresponding approximate identity $(xr_{j})_{j\in J}$, resp. $(l_{j}x)_{j\in J}$, is norm bounded (operator norm bounded), then we speak about \textit{boundedly} (\textit{op-boundedly}) approximately right, resp. left invertible element $x\in\A$. Clearly, zero element in $\A\ne\{0\}$ cannot be approximately invertible in $\A$.

	As far as we know the concept of approximate invertibility is not included in any available literature we have seen, although it seems very natural in the context of algebras with approximate identities. A closely related notion is given by Thatte and Bhatt~\cite{ThatteBhatt} who defined the concept of \textit{topological invertibility in unital topological algebras}. This concept was further developed by Akkar et al.~\cite{Akkar-Beddaa-Oudadess} with improving some proofs. For the sake of consistency with our considerations we recall it here in the context of normed algebras only. 
	
	\begin{definition}
		An element $x$ in a unital normed algebra $\A$ (with unit $e$) is called \textit{topologically right invertible} in $\A$, if there is a net $(r_j)_{j\in J}$ in $\A$ such that $xr_j \to  e$. Similarly, $x\in\A$ is called  \textit{topologically left invertible} in $\A$, if there is a net $(l_j)_{j\in J}$ in $\A$ such that $l_jx \to e$. 
	\end{definition}
	
	\begin{example}
		Arens' algebra $$L^w([0,1])\eqdef \bigcap_{1\leq p\leq +\infty} L_p([0,1])$$ is a unital complete metrizable algebra with pointwise operations and the topology of $L_p$-convergence for each $1\leq p<+\infty$. Then $f(x)=x$ is not invertible in $L^w([0,1])$, but it is topologically invertible, because there is a sequence $g_n(x) = \chi_{[1/n,1]}(x)\frac{1}{x}\in L^w([0,1]), \, n\in\mathbb{N},$ such that $\|fg_n-1\|_p \to 0$ for each $p$. Also, this sequence serves to show that $f$ is approximately invertible as well. In fact, in unital algebras both concepts coincide.
	\end{example}
	
	\begin{proposition}\label{prop:TopInv<->AppInv}
		An element $x$ in a unital normed algebra $\A$ is approximately right invertible if and only if it is topologically right invertible.
	\end{proposition}
	
	\begin{proof} Let $x$ be an approximately right invertible element in $\A$ (with unit $e$). Thus, there is a net $(r_j)_{j\in J}$ in $\A$ such that for each $z\in \A$ we have $xr_jz\to z$. Then for $z=e$ we get topological right invertibility of $x$. 
		
		If $x$ is topologically right invertible in $\A$, then there is a net $(r_j)_{j\in J}$ in $\A$ such that $\|xr_j-e\|\to 0$. Then for each $z\in\A$,  we have $\|xr_jz-z\| = \|(xr_j-e)z\|\leq \|xr_j-e\|\cdot\|z\|.$ Thus, $\lim_{j}\|xr_jz-z\|=0$.  Similarly, $\|zxr_j-z\|\to 0$. Therefore, $x$ is approximately right invertible in $\A$. 
	\end{proof}
	
	For a normed algebra $\A$ we denote by $\Inv_r(\A)$ the set of all right invertible elements in $\A$, by $\TopInv_r(\A)$ the set of all topologically right invertible elements in $\A$, and by $\AppInv_r(\A)$ the set of all approximately right invertible elements in $\A$, respectively. Their left versions are denoted analogously.
	It is well-known from~\cite{ThatteBhatt} that in a unital Banach algebra topological invertibility coincides with invertibility. This implies that in a unital Banach algebra $\A$ we have
	$ \Inv_r(\A)=\TopInv_r(\A)=\AppInv_r(\A).$
	
	Finally, the concept of approximate right (left) invertibility generalizes topological right (left) invertibility to the case of non-unital algebras, where the concepts of invertible and topologically invertible elements are not defined. Indeed, 
	\begin{itemize}
		\item if $\A$ is a non-unital normed (or topological) algebra,
		then $\Inv_r(\A)=\TopInv_r(\A)=\emptyset\subseteq\AppInv_r(\A)$;
		\item if $\A$ is a unital normed (or topological) algebra,
		then $\Inv_r(\A)\subseteq\TopInv_r(\A)=\AppInv_r(\A)$;
		\item if $\A$ is a unital Banach algebra, then
		$\Inv_r(\A)=\TopInv_r(\A)=\AppInv_r(\A)$.
	\end{itemize}
	From this point of view we are mainly interested in non-unital algebras throughout this paper to investigate properties of approximate invertibility in detail.

	In the literature one can find yet another related concept -- a topological quasi-invertibility, see~\cite{Najmi, Zohri-Jabbari-2010}. Since we are working with normed algebras $\A$, we give the definition in that context although it can be stated in the context of topological algebras in general. Given $a,b\in\A$, we define the "circle operation" $a\circ b:=ab-a-b$. 
	
	\begin{definition}\rm
		An element $a$ of a 
		normed algebra $\A$ is called \emph{right quasi-invertible} 
		if there exists an element $b\in\A$ such that $a\circ b=0$. An element $a\in\A$  
		is \emph{topologically right quasi-invertible} 
		if there exists a net $(b_j)_{j\in J}$
		in $\A$ such that the net $(a\circ b_j)_{j\in J}$ converges to the zero element of $\A$. We denote the set of all topologically right quasi-invertible elements of $\A$ by $\TopQInv_r(\A)$.
	\end{definition}
	
	Note that the zero element of an algebra $\A$ always belongs to $\TopQInv_r(\A)$. It should be noted here that the set of topologically quasi-invertible elements had also been considered by M. Abel~\cite{Abel} already in 2001, however from another point of view. Indeed, Abel provided a characterization of topological algebras in which  the set of topologically quasi-invertible elements coincides with the set of quasi-invertible elements. Recently, Abel and Z\'{a}rate-Rodr\'{i}guez~\cite{A-ZR} studied the properties of left, right, and two-sided topologically quasi-invertible elements showing that all these sets are $G_\delta$-sets in F-algebras $\A$.
	
	A connection between $\TopQInv_r(\A)$ and $\TopInv_r(\A)$ for a unital normed algebra $\A$ is well-known. Let us mention that in a unital algebra $\A$ with unit $e$ the equation $a\circ b=0\Leftrightarrow (e-a)(e-b)=e$ holds for each $a,b\in\A$. This implies that $x$ is right quasi-invertible if and only if $e-x$ is right invertible. A topological version then reads as follows.

	\begin{proposition}{\cite[Proposition 2.2(i)]{Zohri-Jabbari-2010}}\label{prop-TopQInvr=e-TopInvr}
		Let $\A$ be a 
		normed algebra with unit $e$. Then 
		$a\in\TopQInv_r(\A)$ if and only if $(e-a)\in\TopInv_r(\A)$. 
	\end{proposition}

	The previous result in a unital algebra $\A$ can be written as $\TopQInv_r(\A)=\{e\}-\TopInv_r(\A)$. In the case of a non-unital algebra $\A$ we can use \cite[Proposition 2.2(iii)]{Zohri-Jabbari-2010} to conclude that $\TopQInv_r(\A)=\TopQInv_r(\A_{1})$, where $\A_{1}$ is the unitization of the normed algebra $\A$. Thus, by Proposition \ref{prop-TopQInvr=e-TopInvr} and remarks above  
	we have
	$\TopQInv_r(\A)=\{(0,1)\}-\AppInv_r(\A_{1}).$

	\subsection{Relation with 
		ideals}\label{sec:ideals}
	
	In the literature we can find a little bit different definition of topological invertibility, see e.g.~\cite{Peimbert-Hoyo}. Indeed, the equalities $\overline{x\A} = \overline{\A x} = \A$, with $x$ being an element of a unital normed algebra $\A$, are taken therein as the  definition of topological invertibility of $x\in\A$. Now we extend these equalities to the case of approximate invertibility. Note that the image (range) of the left multiplication operator $L_x$ is the right principal ideal
	generated by $x$, i.e., $\operatorname{Ran}(L_x) = x\A$.
	Now, we show that the approximate right invertibility of $x$
	is closely related to the density of $x\A$ in $\A$.

	\begin{theorem}\label{thm:dense_ideals_implies_ainv}
		Let $\A$ be a normed algebra with an approximate identity and $x\in\A$.
		Then $x\A$ is dense in $\A$, if and only if $x\in\AppInv_r(\A)$.
	\end{theorem}
	
	\begin{proof}
		Suppose that $x\A$ is dense in $\A$. Then by \cite[Lemma 1.4]{Doran-Wichmann}
		we can find an approximate identity $(\aid_j)_{j\in J}$ with values in $x\A$.
		It means that $\aid_j$ may be written as $xr_{j}$
		with some $r_{j}\in\A$,
		and by definition of approximate invertibility it means that
		$x\in\AppInv_r(\A)$.
		
		Conversely, if  $x\in\AppInv_{r}(\A)$ then there is a net $(r_{j})_{j\in J}$   in $\A$
		such that $(xr_{j})_{j\in J}$ is an approximate identity in $\A$.
		Hence, given $z\in\A$, the net $(xr_{j}z)_{j\in J}$
		takes values in $x\A$ and converges to $z$. 
		Therefore, $x\A$ is dense in $\A$.
	\end{proof}

	Note that if in a normed algebra $\A$  with an approximate identity there exists a principal dense ideal, namely $x_{0}\A$,  then by Theorem \ref{thm:dense_ideals_implies_ainv} we have  $x_{0}\in\AppInv(\A)$, i.e., $\AppInv(\A)\neq\emptyset.$

	\begin{proposition}\label{prop:x-oveAx}
		Let $\A$ be a  
		normed algebra with an approximate identity.  
		Then $x\in\overline{x\A }$.
	\end{proposition}
	
	\begin{proof}
		Let $(e_{j})_{j\in\,J}$ be an approximate identity in $\A$. Then, the net  $(xe_{j})_{j\in\,J}$ takes values in $x\A $ and hence $\displaystyle x  = \lim_{j\in J}xe_{j}\in\overline{x\A }$.
	\end{proof}
	
	\begin{proposition}
		Let $\A$ be a non-unital normed algebra and
		$x\in\AppInv_r(\A)$.
		Then $x\notin\A x$.
	\end{proposition}
	
	\begin{proof}
		Suppose that $x\in\A x$.
		Then there is an element $y\in\A$ such that $x=yx$.
		On the other hand, since $x$ is approximately right invertible,
		there is a net $(r_{j})_{j\in\,J}$ in $\A$
		such that $(xr_{j})_{j\in\,J}$ is an approximate identity in $\A$.
		Now, $y=\lim\limits_{j\in\,J} yxr_{j}=\lim\limits_{j\in\,J} xr_{j}$.
		By Proposition~\ref{prop:convergentAI}, $y$ is an identity in $\A$,
		which contradicts the hypothesis.
	\end{proof}

	\begin{theorem}\label{conj-app=A/I}
		Let $\A$ be a non-unital normed algebra with an approximate identity   
		and $\mathcal{I}_{\hspn r}(\A)$ be the set of all closed proper right ideals in $\A$. Then \begin{equation}\label{app=au}
			\AppInv_{r}(\A)=  \A\setminus \bigcup_{I\in\mathcal{I}_{\hspn r}(\A)}\hspace{-5pt} I.
		\end{equation}
	\end{theorem}
	\begin{proof}
		Suppose that $x\in\AppInv_{r}(\A)$ and assume there exists $I_{x}\in \mathcal{I}_r(\A)$ such that $x\in I_{x}$. Thus, $x\A\subset I_{x}$ and, in consequence, $\A=I_{x}$ because  $x\A$ is dense in $\A$ by Theorem \ref{thm:dense_ideals_implies_ainv}, which yields a contradiction. Then $\displaystyle x\in \A\setminus \bigcup_{I\in\mathcal{I}_r(\A)}\hspace{-5pt}I$.

		On the other hand, let $\displaystyle x\in\A\setminus \bigcup_{I\in\mathcal{I}_r(\A)}\hspace{-5pt} I$  and assume that $x\notin\AppInv_{r}(\A)$. Therefore, $\overline{x\A}\in\mathcal{I}_{r}(\A)$  since $x\A$ is not dense in $\A$ by Theorem \ref{thm:dense_ideals_implies_ainv}. Moreover, 
		$x\in\overline{x\A}$ by Proposition \ref{prop:x-oveAx}. This contradicts our assumption, and hence  $x\in\AppInv_{r}(\A)$, i.e., 
		\eqref{app=au} holds.
	\end{proof}

	In a unital normed algebra $\A$ every proper left [right, two-sided] ideal is contained in a maximal left [right, two-sided] ideal. It is a crucial tool in studying normed algebras with identity, but it is not valid for algebras without identity. Next we study the regular, or modular ideals in normed algebras without identity.
	For more details see \cite[Section 1.1]{Larsen}.
	
	\begin{definition}
		Let $\A$ be an algebra. A left [right; two-sided] 
		ideal $I$ is said to be \emph{modular}, if there exists some $v\in\A$ such that $xv-x\in\,I$ [$vx-x\in\,I$; $xv-x\in\,I$ and $vx-x\in\,I$] 
		for all $x\in\A$. The element $v$ is called a left [right; two-sided] \emph{identity modulo} $I$.
	\end{definition}
	
	\begin{proposition}\label{prop:ainv_modid}
		Let $\A$ be a non-unital normed algebra.
		If $x\in\AppInv_r(\A)$,
		then there is no maximal modular right ideal $J$ such that $x\in\,J$.  
	\end{proposition}
	
	\begin{proof}
		If there is a maximal modular right ideal $J\subsetneq\A$
		such that $x\in\,J$, then $x\A\subset\,J$.
		However, every maximal modular right ideal is closed,
		see \cite{Larsen}, then by Theorem \ref{thm:dense_ideals_implies_ainv} we have $\A=\overline{x\A}\subset\,J$.
		Thus, $J=\A$, which is a contradiction with the maximality of $J$.
		Hence, such a maximal modular right ideal does not exist.  
	\end{proof}

	Given a proper closed right  ideal $I$ in $\A$,
	we denote by $\Hull_r(I,\A)$ the set of all
	maximal right modular ideals containing $I$.   The converse implication of Proposition \ref{prop:ainv_modid} follows easily in normed algebras with ``rich'' sets of maximal right modular ideals, i.e., $\Hull_{r}(I,\A)\neq\emptyset$. In that case, the approximate right invertibility may be described as follows.

	\begin{proposition}\label{prop:rich_modular_ideals}
		Let $\A$ be a non-unital normed algebra $\A$ with an approximate identity such that $\Hull_r(I,\A)\ne\emptyset$
		for every proper closed right  ideal $I$ in $\A$.
		Then the following conditions are equivalent:
		\begin{enumerate}
			\item $x\in\AppInv_r(\A)$,
			\item $x\A$ is dense in $\A$,
			\item for every maximal right modular ideal $J$ in $\A$ we have
			$x\notin J$.
		\end{enumerate}
	\end{proposition}
	\begin{proof}The equivalence (i)$\Leftrightarrow$(ii) and the implication (ii)$\Rightarrow$(iii) were proven before in Theorem \ref{thm:dense_ideals_implies_ainv} and
		Proposition \ref{prop:ainv_modid}, respectively.
		
		Suppose that (ii) is not true. Then $\overline{x\A}$ is a proper closed right  ideal of $\A$. Let $J\in \Hull_r(\overline{x\A},\A)$,  then $x\in\overline{x\A}\subset J$ by Proposition \ref{prop:x-oveAx}.
	\end{proof}
	
	If $\A$ is a radical algebra,
	then the set $\Hull_{r}(I,\A)$ is empty. For the group algebra $\Lone(G)$ with a non-discrete locally compact abelian topological group $G$ and $C_{0}(X)$ with a locally compact Hausdorff topological space $X$, the set $\Hull_{r}(I,\A)$ is nonempty, see \cite[\textsection 8.7]{Larsen}.
	More general examples include the regular semisimple tauberian abelian Banach algebras, see~\cite[Lemma 5.1.9]{Kaniuth}, \cite[\textsection 8.7]{Larsen}.

	\subsection{Elementary properties}

	This section summarizes some elementary properties of approximate invertibility. The first results connect approximate invertibility with convergence of the corresponding net and construction of approximate identities.
	
	\begin{proposition}\label{prop:AppInv->divergent_net}
		Let $\A$ be a non-unital normed algebra. If $x\in\AppInv_r(\A)$, then the corresponding net $(r_j)_{j\in J}$ in $\A$ is not convergent.
	\end{proposition}
	
	\begin{proof}
		Suppose that $(r_j)_{j\in J}$ is convergent in $\A$, i.e., there exists an element $r\in\A$ such that $r_j\to r$ in $\A$. Since $x\neq 0$, then the continuity of multiplication yields that $xr_j\to xr$. Moreover, $x\in\A$ is approximately right invertible, which means that the net $(xr_j)_{j\in J}$ is an approximate identity in $\A$ with the limit $xr$. Thus, Proposition~\ref{prop:convergentAI} states that $\A$ is unital, which is a contradiction.
	\end{proof}

	\begin{proposition}\label{Proposition:net}
		Let $\A$ be a normed algebra and $x\in\A$ be boundedly approximately right and left invertible. Then there exists a net $(w_k)_{k\in K}$ such that the net $(xw_{k}x)_{k\in\,K}$ is a bounded approximate identity in $\A$. 
	\end{proposition}
	
	\begin{proof}
		Let $(l_i)_{i\in I}$ and $(r_j)_{j\in J}$ be the nets
		such that $(l_i x)_{i\in I}$ and $(x r_j)_{j\in J}$
		are bounded approximate identities in $\A$. By~\cite[Corollary 7, Chapter 6]{Feichtinger}  
		the net $(xr_{j}l_{i}x)_{(i,j)\in I\times J}$ is a bounded approximate identity in $\A$. Now it is enough to put $w_{(i,j)}=r_jl_i$ for $(i,j)\in I\times J$.		  
	\end{proof}
	
	\begin{remark} 
		Moreover, if the assumption of the latter proposition is fulfilled, then
		for each $z\in\A$ it holds \begin{equation}\label{limu_pv_q}\lim_{(i,j)\in I\times J} l_ixzxr_j = z,
		\end{equation} where $(l_i)_{i\in\,I}$ and $(r_j)_{j\in\,J}$ are the corresponding nets from the proof of Proposition~\ref{Proposition:net}. The proof of~(\ref{limu_pv_q}) is based on the inequality $\|l_ixzxr_j-z\|\leq \|l_ixz-z\|\cdot\|xr_j\|+\|zxr_j-z\|.$ 
		By \cite[Proposition 2.6]{Doran-Wichmann} the net $(l_{i}x+xr_{j}-xr_{j}l_{i}x)_{(i,j)\in I\times J}$ is a bounded approximate identity in $\A$.
	\end{remark}
	
	The following result describes a connection of approximate invertibility with multiplication operator. 
	
	\begin{lemma}\label{left-multi}
		Let $\A$ be a normed algebra. If $x\in\A$ is boundedly or op-boundedly approximately left invertible in $\A$,  then the left 
		multiplication operator $L_x\colon \A\to\A$ given by $L_{x}(y)=xy$ for $y\in\A$, is bounded.
	\end{lemma}
	
	\begin{proof}
		Assume   $x$ is boundedly, but not op-boundedly,  approximately left invertible in $\A$. Then there exists a net $(l_{j})_{j\in J}$ such that $(l_{j}x)_{j\in J}$ is an approximate identity in $\A$ with $\|l_{j}x\|\leq M$ for all $j$ and some $M>0$.   Thus, $\|L_{x}(l_{j}xy)\|\leq M\,\| x\|\,\|y\|$ for every $j$ and every $y\in\A$. Now, by continuity of the product and the norm $\|\cdot\|$ in $\A$ one has $\displaystyle\lim_{j\in J}\|L_{x}(l_{j}xy)\|=\|L_{x}(y)\|$. Hence, $\|L_{x}(y)\|\leq K\|y\|$ for every $y\in\A$ with $K=M\|x\|$.
		
		Now, assume $x$ is op-boundedly, but not boundedly,  approximately left invertible in $\A$. Then there exists a net $(l_{j})_{j\in J}$ such that $(l_{j}x)_{j\in J}$ is an approximate identity in $\A$ with $M=\sup_{j\in J}\|T_{j}\|_{\rm{op}}<\infty$, where $\|T_{j}\|_{\rm{op}}=\sup_{\|y\|\leq 1}\|l_{j}xy\|$ for all $j$.   Thus, for every $j$ and every $y\in\A$ we have, $$\|L_{x}(l_{j}xy)\|=\|xT_{j}y\|\leq \| x\|\,\|T_{j}y\|\leq \|x\|\,\|T_{j}\|_{\rm{op}}\,\|y\|\leq M\,\|x\|\,\|y\|.$$ Now, by continuity of the product and the norm $\|\cdot\|$ in $\A$ one has $\displaystyle\lim_{j\in J}\|L_{x}(l_{j}xy)\|=\|L_{x}(y)\|$. Hence, $\|L_{x}(y)\|\leq K\|y\|$ for every $y\in\A$ with $K=M\|x\|$.
	\end{proof}
	
	Now, we prove that the set of all boundedly approximately invertible elements of an abelian algebra is closed with respect to (finitely many elements) multiplication. 
	
	\begin{proposition}
		\label{prop:bounded_approx_inv_of_product}
		Let $\A$ be an abelian normed algebra, and $x_1, \dots, x_n\in\A$, $n\in\mathbb{N}$. Then $x_1\cdots x_n$ is boundedly approximately invertible in $\A$ if and only if for each $k$ in $\{1,\ldots,n\}$ the element $x_k$ is boundedly approximately invertible in $\A$. 
	\end{proposition}
	
	\begin{proof}
		Let $x_1\cdots x_n$ be boundedly approximately invertible in $\A$. Then there exists a net $(u_i)_{i\in I}$ in $\A$ such that $(x_1\cdots x_nu_i)_{i\in I}$ is a bounded approximate identity in $\A$, i.e., there is a constant $K>0$ such that $\|x_1\cdots x_n u_i\|\leq K$ for each $i\in I$, and for each $z\in\A$ it holds $x_1\cdots x_n u_i z \to z$. By  the commutativity, it is sufficient to show that $x_1$ is boundedly approximately invertible in $\A$. For each $i$ in $I$ we put $w_i = x_2\cdots x_n u_i$. Clearly, $\|x_1w_i\|\leq K$ for each $i\in I$, and for each $z\in\A$ the net $(x_1w_iz)_{i\in I}$ converges to $z$, i.e., $(x_1w_i)_{i\in I}$ is a bounded approximate identity in $\A$. Thus, $x_1$ is boundedly approximately invertible in $\A$.

		In the converse direction, it is sufficient to treat the case $n=2$ only, the rest will proceed by induction. Let $x_{1}$ and $x_{2}$ be boundedly approximately invertible in $\A$, i.e., there exist nets $(u_i)_{i\in I}$ and $(v_j)_{j\in J}$ in $\A$ and constants $K,M>0$ such that $\|x_1u_i\|\leq K$ for each $i\in I$ and $\|x_2v_j\|\leq M$ for each $j\in J$. The mapping $B_{x_{1}}(x,y)=x_{1}yx$  is a  bilinear form with $\|B_{x_{1}}\|\leq K \|x_{1}\|$ by Lemma \ref{left-multi}. Now, by \cite[Proposition 10]{Feichtinger} for any $z\in\A$ we have 
		\[\lim_{(i,j)\in I\times J}x_{1}x_{2}u_{i}v_{j}z=\lim_{i\in I}\lim_{j\in J}B_{x_{1}}(u_{i}z,v_{j}x_{2})=\lim_{i\in I}\lim_{j\in J}(x_{1}u_{i}z)(x_{2}v_{j})=z,\]
		i.e., $x_{1}x_{2}$ is approximately invertible in $\A$. 
	\end{proof}
	
	An analogue to Proposition~\ref{prop:bounded_approx_inv_of_product} also holds for op-boundedly approximately invertible elements.
	
	\begin{proposition}
		\label{prop:opbounded_approx_inv_of_product}
		Let $\A$ be an abelian normed algebra, and $x_1,\dots,x_n\in\A$, $n\in\bN$.
		Then $x_1\cdots x_n$ is op-boundedly approximately invertible in $\A$ if and only if
		for each $k$ in $\{1,\ldots,n\}$ the element $x_k$ is op-boundedly approximately invertible in $\A$.
	\end{proposition}
	
	\begin{proof}
		The proof is almost the same as the proof of Proposition~\ref{prop:bounded_approx_inv_of_product}.
		In the second part,
		we suppose that $\|L_{x_1 u_i}\|_{\text{op}}\le K$
		for each $i$ in $I$
		and $\|L_{x_2 v_j}\|_{\text{op}}\le M$ for each $j$ in $J$.
		Next, we apply the upper bound
		$\|L_{x_1 x_2 u_i v_j}\|_{\text{op}}
		=\|L_{x_1 u_i} L_{x_2 v_j}\|_{\text{op}}
		\le K M$.
	\end{proof}

	\begin{proposition}\label{prop:right_zero_divisor_is_not_appinvr}
		Let $\A$ be a non-unital normed algebra.
		If $x\in\AppInv_r(\A)$,
		then $x$ is not a right zero divisor in $\A$.
	\end{proposition}
	
	\begin{proof}
		Suppose that $x\ne0$ and there exists $y\in\A$ such that $y\ne0$ and $yx=0$.
		If $(r_j)_{j\in J}$ is a net in $\A$, then $yxr_j=0$ for $j\in J$.
		So, the net $(yxr_j)_{j\in J}$ cannot converge to $y$, and $x\notin\AppInv_r(\A)$.
	\end{proof}

	Finally, we mention a relation with topological divisors of zero. Recall that an element $x$ of an algebra $\A$ is called a \textit{left} (resp. \textit{right}) \textit{topological divisor of zero} if there is a net $(l_j)_{j\in J}$ 
	(resp. $(r_j)_{j\in J}$) in $\A$
	such that the net $(xl_j)_{j\in J}$ (resp. $(r_j x)_{j\in J}$) converges to $0$ whereas neither $(l_j)_{j\in J}$ nor $(r_j)_{j\in J}$ does not converge to $0$.
	
	Different characterizations of topological divisors of zero may be found in~\cite[Theorem~1.6.2]{Larsen}.

	As it is well-known, in unital Banach algebras topological divisors of zero are “fundamentally” non-invertible (i.e., there is no larger normed algebra in which they might become invertible). Recently, Schulz, Brits, and Hasse have shown that in non-unital Banach algebras admitting a two-sided (left, resp. right) not necessarily bounded approximate identity each element is a two-sided (right, resp. left) topological divisor of zero, cf. \cite[Theorem 1.2]{SchulzBritsHasse2017}. Thus, the relationship between approximately invertible elements and topological divisors of zero in non-unital Banach algebras is an easy consequence. However, we provide an alternative (and independent) proof of this result.
	
	\begin{proposition}\label{AppInvR->TdzL}
		Let $\A$ be a non-unital Banach algebra. Then every approximately right invertible element in $\A$ is a left topological divisor of zero.
	\end{proposition}
	
	\begin{proof}
		If $x\in\AppInv_r(\A)$, then there is a net $(r_j)_{j\in J}$  in $\A$ such that $(xr_j)_{j\in J}$ is an approximate identity in $\A$. 
		Suppose that $x$ is not a left topological divisor of zero. Then $\zeta_l(x) = \inf\limits_{y\in\mathbb{S}(\A)} \|xy\| > 0$, see  \cite[Theorem 1.6.2, page 46]{Larsen}. Without loss of generality we may assume that $\zeta_l(x)\geq 1$ (replacing $x$ by $x/\zeta_l(x)$), which is equivalent to the assertion that \begin{equation}\label{eq:inequality}
			\|xy\|\geq \|y\|\,\,\, \textrm{for each}\,\, y\in\A.
		\end{equation} Since $(xr_j)_{j\in J}$ is an approximate identity in $\A$ and
		$$
		\|(r_i-r_j)x\| \leq \|x(r_i-r_j)x\|=\|(xr_ix-x) - (xr_jx-x)\| \leq \|(xr_i)x-x\| + \|(xr_j)x-x\|,
		$$
		the net $(r_jx)_{j\in J}$ is Cauchy in $\A$.
		From the completeness of $\A$ it results that $(r_jx)_{j\in J}$ converges in $\A$, say to $u\in\A$.
		Immediately, $\displaystyle \lim_{j\in J}x(xr_{j})=x=\lim_{j\in J} (xr_{j})x
		=\lim_{j\in J} x(r_{j}x)=xu,$
		but $\|xr_j-u\|\le\|x(xr_{j})-xu\|$ by \eqref{eq:inequality},
		therefore $(xr_j)_{j\in J}$ converges to $u$.
		By Proposition~\ref{prop:convergentAI}, $u$ is the unit in $\A$,
		which contradicts the assumption that $\A$ is not unital. Therefore, we conclude that $x$ is a left topological divisor of zero.  
	\end{proof}

	\begin{remark}
		Clearly, in non-unital abelian Banach algebras the above result implies that each approximately invertible element is a topological divisor of zero. In non-commutative case we need to assume that $x$ is approximately right and left invertible in order $x$ to be a two-sided topological divisor of zero.
	\end{remark}

	
	\section{Approximate invertibility in some classes of algebras}\label{sec:AppInv in algebras}
	
	Here we provide a detailed study of approximate invertibility in some classes of algebras such as Banach algebras, C*-algebras and involutive algebras.

	\subsection{Non-unital abelian Banach algebras}
	
	Invertibility of elements of an abelian Banach algebra is a matter of Gelfand's spectral theory.  Here we study approximate invertibility and its connection with Gelfand's transform for the case of non-unital abelian Banach algebras.
	
	The Gelfand theory of non-unital abelian Banach algebras is explained, for example, in~\cite[Section 1.3]{Murphy} and \cite[Section~3.1]{Larsen}. 
	
	For an algebra $\A$ denote by $\Characters{\A}$ the set of its \emph{characters}, i.e., non-zero algebra homomorphisms $\A\to\mathbb{C}$. The set $\Characters{\A}$ is a subset of the closed unit ball of the dual space of $\A$.
	This set is equipped with the relative weak-* topology and may be identified one-to-one with the set $\{M\colon M\,\text{is maximal modular ideal of $\A$}\}$ by means of the mapping $\varphi\mapsto\ker\varphi$.

	Given $x\in\A$, we denote by $\hat x$ the function $\Characters{\A}\to\mathbb{C}$
	defined by $\hat{x}(\phi)=\phi(x)$ for every $\phi\in\Characters{\A}$.
	The correspondence $x\mapsto\hat{x}$ is called the \emph{Gelfand transform}.
	It is well known that the Gelfand transform is a homomorphism of $\A$
	into $C_0(\Characters{\A})$. Moreover, if $\|\cdot\|_{\infty}$ denotes the sup-norm on  $\mathcal{C}_{0}(\mathcal{M}_{\A})$, then $\|\hat{x}\|_{\infty}\leq\|x\|$ for every $x\in\A$.
	
	\begin{proposition}
		\label{prop:aid_converges_pointwisely_to_one}
		Let $\A$ be a non-unital abelian Banach algebra
		and $(\aid_{j})_{j\in J}$ be an approximate identity in $\A$.
		Then $ \lim_{j\in J} \hat{\aid}_{j}(\phi)=1$ for every $\phi\in\mathcal{M}_{\A}$.
		
	\end{proposition}
	
	\begin{proof}
		Let $\phi\in\mathcal{M}_{\A}$.
		Select $x\in\A$ such that $\phi(x)\ne0$, i.e. $\hat{x}(\phi)\ne0$.
		Since $\phi$ is a multiplicative functional on $\A$, we have
		$
		|\hat{x}(\phi)|\,|\hat{\aid}_j(\phi)-1|
		=|\phi(x)\phi(\aid_j)-\phi(x)|
		=|\phi(x\aid_j-x)|
		\le\|x\aid_j-x\|.
		$
		Since the net $(\|x\aid_j-x\|)_{j\in J}$ tends to $0$, the net $(\hat{\aid}_j(\phi)-1)_{j\in J}$  tends to zero as well.
	\end{proof}

	Since every maximal modular ideal $I$ of a non-unital abelian Banach algebra $\A$ is of the form $I=\ker(\varphi)$ for some $\varphi\in\Characters{\A}$, Theorem \ref{thm:dense_ideals_implies_ainv} and Proposition \ref{prop:ainv_modid} may be summarized as follows.
	
	\begin{proposition}\label{prop:ainv_situation_in_cba}
		Let $\A$ be a non-unital abelian Banach algebra with an approximate identity 
		and $x\in\A$. The following statements hold:
		\begin{enumerate}
			\item[\rm{(a)}] $x\in\AppInv(\A)$ if and only if
			$x\A$ is dense in $\A$.
			\item[\rm{(b)}] If $x\in\AppInv(\A)$, then there is no maximal modular ideal $J$ such that $x\in J$.
			\item[\rm{(c)}] If $x\in\AppInv(\A)$,
			then $\hat{x}(\varphi)\ne0$ for every
			$\varphi$ in $\Characters{\A}$.
		\end{enumerate}
	\end{proposition}


	\begin{remark}
		By Proposition \ref{prop:rich_modular_ideals} the reverse implication in (b) need not hold in general, see the small disk algebra in  Section~\ref{Section:SmallDiskAlgebra}. However, there are some classes of non-unital abelian Banach algebras where the converse holds true, e.g., non-unital abelian C*-algebras, see Section~\ref{Section:AbelianC*-algebras}. 
		More general, the reverse implication in (b) holds for the regular semisimple tauberian abelian Banach algebras, see~\cite[Lemma 5.1.9]{Kaniuth}. A general characterization of all these algebras is yet unknown for us.
	\end{remark}
	
	The next result is a consequence of  Proposition \ref{prop:ainv_situation_in_cba} saying that the set $\AppInv(\A)$ is a proper subset of $\A\setminus\{0\}$ provided $\Characters{\A}\neq\emptyset$.
	\begin{corollary}\label{cor-app=A/0}
		Let $\A\neq\{0\}$ be a non-unital abelian Banach algebra with an approximate identity 
		such that $\A\setminus\{0\}=\AppInv(\A)$. Then $\Characters{\A}=\emptyset$, i.e., there are no proper modular ideals in $\A$.
			\end{corollary}

	\subsection{Non-unital, non-abelian C*-algebras and maximal modular 
		ideals}\label{Section:nonAbelianC*-algebras}
	
	Next, we study approximately invertible elements in C*-algebras by means of maximal modular right ideals and pure states.
	It is well-known that in a C*-algebra $\A$ the following identity holds: 
	\begin{equation}\label{norminvolution}
		\|x^{*}\|=\|x\|,\quad\, x\in\A.
	\end{equation}
	However, since in any normed algebra $\A$ with a continuous involution $x\rightarrow\,x^{*}$ there is an equivalent submultiplicative norm satisfying the condition \eqref{norminvolution}, see \cite{Zelazko}, we may consider the mentioned relationship in algebras $\A$ with involutions and norms   satisfying \eqref{norminvolution}. The relationship between an element approximately right (left) invertible and its conjugate is the following.
	
	\begin{proposition}\label{AppInvr-AppInvl}
		Let $\A$ be a normed $*$-algebra. Then $x\in\AppInv_{r}(\A)$ if and only if $x^{*}\in\AppInv_{\ell}(\A)$.
	\end{proposition}
	
	\begin{proof}
		If $x\in\AppInv_{r}(\A)$, then there is a net $(r_{j})_{j\in J}$ in $\A$ such that for all $y\in\A$ we have $\lim_{j\in J}\|xr_{j}y-y\|=\lim_{j\in J}\|yxr_{j}-y\|=0$. Hence, by \eqref{norminvolution}, the net $(l_{j}=r_{j}^{*})_{j\in J}$, is such that 
		\begin{align*}
			\lim_{j\in J}\|l_{j}x^{*}y-y\|&=\lim_{j\in J}\left\|\left(l_{j}x^{*}y-y\right)^{*}\right\|
			=\lim_{j\in J}\|y^{*}xr_{j}-y^{*}\|=0\\
			&=\lim_{j\in J}\|yl_{j}x^{*}-y\|=\lim_{j\in J}\|xr_{j}y^{*}-y^{*}\|, \quad y\in\A.
		\end{align*}
		That is, $(l_{j}x)_{j\in J}$ is an approximate identity in $\A$ proving that $x^{*}\in\AppInv_{\ell}(\A)$. The same argument provides the sufficiency.
	\end{proof}
	
	It is well-known that every C*-algebra possesses an approximate identity (see, for instance \cite[Theorem 3.1.1]{Murphy}).
	We denote by $\operatorname{PS}(\A)$ the set of all pure states on $\A$.
	For any $\tau\in\operatorname{PS}(\A)$, the set
	\begin{equation}\label{Nt}
		N_{\tau}=\{a\in\A\colon \tau(a^{*}a)=0\}
	\end{equation}
	is a maximal modular left ideal of $\A$,
	and $\tau\mapsto N_\tau$ is a bijective correspondence
	between the pure states and the modular maximal left ideals of a C*-algebra, see for instance \cite[Theorem 5.3.5]{Murphy}. Then the criterion for approximate invertibility in a non-unital C*-algebra reads as follows. 
	
	\begin{theorem}\label{Appinvl-modi-C*}
		Let $\A$ be a non-unital non-abelian C*-algebra, and $x\in\A$. Then  the following statements are equivalent:
		\begin{enumerate}
			\item $x\in\AppInv_{r}(\A)$.
			\item  $x$ does not belong to any maximal modular right ideal.
			\item  $\tau(xx^{*})\neq 0$ for all $\tau\in \operatorname{PS}(\A)$.
		\end{enumerate}
	\end{theorem}
	
	\begin{proof} 
		By Proposition~\ref{prop:ainv_modid},
		(i) implies (ii) and the equivalence (ii)$\Leftrightarrow$(iii) follows from \cite[Theorem 5.3.5]{Murphy}, so
		we only have to prove that (iii) implies (i).

		Suppose that (i) is not true. Then, by Theorem  \ref{thm:dense_ideals_implies_ainv} $x\A $ is not dense in $\A$ and hence $\overline{\A x^{*}}$ is a proper left ideal of $\A$ by Proposition \ref{AppInvr-AppInvl}  and 
		 
		$x^{*}\hspace{-5pt}\in\hspace{-3pt}\overline{\A x^{*}}$ by Proposition \ref{prop:x-oveAx}. 
		However, 
		by~\cite[Theorem 5.3.3]{Murphy}, there exists $\tau \hspace{-3pt}\in \hspace{-3pt}\operatorname{PS}(\A)$ such that $\overline{\A x^{*}}\hspace{-3pt}\subseteq \hspace{-3pt} N_\tau$.
		So, $x^{*}\hspace{-3pt}\in\hspace{-3pt} N_\tau$ meaning that (iii) is not true.
	\end{proof}
	
	\subsection{Abelian C*-algebras}\label{Section:AbelianC*-algebras}
	
	In abelian C*-algebras $\A$, any pure state is a character, see for instance \cite[Theorem 5.1.6]{Murphy}, thus Theorem \ref{Appinvl-modi-C*} may be reformulated accordingly providing that the condition $\hat{x}(\phi)\ne0$ for every $\phi\in\Characters{\A}$
	is not only necessary, but also sufficient
	for the approximate invertibility of $x$. By famous Gelfand--Naimark theorem, every abelian C*-algebra $\A$
	is isometrically isomorphic to $C_0(\Characters{\A})$.
	Therefore, without loss of generality, we may only consider the case $\A=C_0(X)$,
	where $X$ is a non-empty locally compact Hausdorff space. Now, we are going to show Theorem \ref{Appinvl-modi-C*} for abelian C*-algebras using only elementary tools and as a-by product we get that the interior of the set $\AppInv(C_{0}(X))$ is empty for every non-compact locally compact Hausdorff space $X$.

	Denote by $\C$ the set of all non-empty compact subsets of $X$
	equipped with the partial order $\subseteq$.
	Note that $\C$ is \emph{directed} by $\subseteq$:
	if $K_1,K_2\in\C$, then $K_1\cup K_2\in\C$,
	$K_1\subseteq K_1\cup K_2$ and $K_2\subseteq K_1\cup K_2$.
	For every $K\in\C$ denote by $1_K$
	the function $K\to\mathbb{C}$
	defined by $1_K(t)=1$ for every $t\in K$.

	\begin{proposition}
		\label{prop:crit_aid_in_csa}
		Let $(\aid_j)_{j\in J}$ be a bounded net in $C_0(X)$.
		Then the following conditions are equivalent:
		\begin{enumerate}
			\item $(\aid_j)_{j\in J}$ is an approximate identity in $C_0(X)$;
			\item for every $K\in\C$,
			the net of restrictions $(\aid_j|_K)_{j\in J}$ converges uniformly to $1_K$.
		\end{enumerate}
	\end{proposition}
	
	\begin{proof}
		(i)$\Rightarrow$(ii).
		Let $(\aid_j)_{j\in J}$ be an approximate unit in $C_0(X)$
		and $K\in\C$. By Uryson's Lemma there exists a function $f\in C_0(X)$
		such that $f|_K=1_K$. Since $(\aid_j)_{j\in J}$ is an approximate unit in $C_0(T)$,
		then $\sup_{t\in X}|f(t) \aid_j(t)-f(t)|\to0$
		and, in particular, $\sup_{t\in K}|f(t)\aid_j(t)-f(t)|\to0$.
		The latter means that $\sup_{t\in K}|\aid_j(t)-1|\to0$.
		
		(ii)$\Rightarrow$(i). Put $\displaystyle M=\sup_{j\in J}\|e_j\|_{\infty}.$ 
		Let $f\in C_0(X)$ and $\eps>0$.
		The case $f\equiv0$ is trivial, so we suppose that $f\nequiv0$.
		Using the definition of $C_0(X)$ choose $K\in\C$
		such that $
		|f(t)|\le\dfrac{\eps}{M+1}
		$ for every $t\in X\setminus K$.
		Since $(\aid_j|_K)_{j\in J}$ converges uniformly to $1_{\hspn K}$,
		find $j_1\hspn\in\hspn J$ such that  $\displaystyle\sup_{t\in K}|\aid_{j}(t)-\hspace{-1pt}1|\hspn<\hspn\frac{\eps}{\|f\|_{\infty}}$, for every~$j\hspn\succ\hspn j_1$.
		Let $j\succ j_1$.
		Then $|f(t)\aid_j(t)-f(t)|
		\le\|f\|_{\infty}\,|\aid_j(t)-1|<\eps,$ for every $t\in K$ and for every $t\in X\setminus K$
		\begin{align*}
			|f(t)\aid_j(t)-f(t)|
			\le |f(t)| |\aid_j(t)-1|
			\le |f(t)| (M+1) < \eps.
		\end{align*} This completes the proof.
	\end{proof}
	
	\begin{corollary} \label{cor:aid_in_csa}
		Let $(\aid_K)_{K\in\C}$ be a net in $C_0(X)$ such that
		$\|\aid_K\|_\infty=1$ for every $K\in\C$
		and $\aid_K(t)=1$ for every $K\in\C$ and every $t\in K$.
		Then $(\aid_K)_{K\in\C}$ is an approximate identity in $C_0(X)$.
	\end{corollary}
	
	\begin{proof}
		Given a compact $K_1\in\C$,
		for every $K\in\C$ with $K_1\subseteq K$
		we have $\aid_K|_{K_1}=1_{K_1}$.
		It means that the net $(\aid_K|_{K_1})_{K\in\C}$
		converges uniformly to $1_{K_1}$.
		By Proposition~\ref{prop:crit_aid_in_csa},
		the net $(\aid_K)_{K\in\C}$
		is an approximate identity in $C_0(X)$.
	\end{proof}
	
	\begin{proposition}[criterion
		for approximate invertibility in abelian C*-algebras]
		\label{prop:criterion_ainv_in_csa}
		Let $X$ be a non-empty locally compact Hausdorff space and $f\in C_0(X)$.
		Then the following conditions are equivalent:
		\begin{enumerate}
			\item  $f$ is approximately invertible;
			\item  $f C_0(X)$ is dense in $C_0(X)$;
			\item  $f(t)\ne0$ for every $t\in X$.
		\end{enumerate}
	\end{proposition}
	
	\begin{proof}
		This is a particular case of Theorem~\ref{Appinvl-modi-C*},
		but we will provide a more elementary proof.
		By  Theorem \ref{thm:dense_ideals_implies_ainv},
		(i) is equivalent to (ii). By Proposition  \ref{prop:ainv_modid} if (ii) holds then there is no maximal modular ideal $J$ of $C_{0}(X)$ such that $f\in J$. However, by \cite[Theorem 3.1.2 and Theorem 4.1.2]{Larsen} all the modular ideals in $C_{0}(X)$ are of the form $J=\left\{g \mid g \in C_{0}(X), g(t)=0\right\}$ for some $t\in X$. Thus $f(t)\neq0$ for all $t\in X$.
		
		On the other hand, let us suppose that (iii) holds and prove (i).
		For every $K\in\C$, applying the Uryson's Lemma,
		we select a function $\aid_K\in C_0(X)$ with values in $[0,1]$
		such that $\aid_K(t)=1$ for every $t\in K$.
		After that we construct $g_K\colon T\to\mathbb{C}$ by the rule
		$g_{K}=\dfrac{\aid_{K}}{f}.
		$
		Then $g_K\in C_0(X)$ and $f g_K=\aid_K$.
		By Corollary~\ref{cor:aid_in_csa},
		$(\aid_K)_{K\in\C}$ is an approximate identity.
	\end{proof}
	
	\begin{proposition}\label{app-inv-C0T}
		Let $X$ be a non-compact locally compact Hausdorff space.
		Then the interior of $\AppInv(C_0(X))$ is empty.
	\end{proposition}
	
	\begin{proof}
		Let $f\in\AppInv(C_0(X))$ and $\eps>0$.
		Choose $K\in\C$ such that
		$|f(t)|<\eps/2$ for every $t\in X\setminus K$.
		Choose an open set $U$ such that $K\subset U$ and $U\ne X$.
		Using the Urysohn’s Lemma for locally compact Hausdorff spaces,
		construct a function $g\in C_0(X)$ such that
		$g(t)=1$ for every $t\in K$, $0\le g(t)\le 1$ for every $t\in X$
		and $\support(g)\subset U$.
		Then $fg\in C_0(X)$,
		$fg$ vanishes on $X\setminus U$ and therefore
		by Proposition~\ref{prop:criterion_ainv_in_csa} $fg\notin\AppInv(C_0(X))$.
		Moreover, $fg$ coincides with $f$ on $K$, and
		\begin{displaymath}
			\|fg-f\|_\infty
			= \sup_{t\in X\setminus K}|f(t)g(t)-f(t)|
			= \sup_{t\in X\setminus K}|1-g(t)|\, |f(t)| \le \frac{\eps}{2}, 
		\end{displaymath}
		which proves the result.
	\end{proof}

	From Proposition~\ref{prop:ainv_situation_in_cba} and Proposition~\ref{prop:criterion_ainv_in_csa} we get the following consequence.
	
	\begin{corollary}\label{HatAppinv-AppinvCo}
		Let $\A$ be a non-unital abelian Banach algebra with an approximate identity such that $\Characters{\A}\neq\emptyset$. Then  $\hat{x}\in\AppInv(C_{0}(\Characters{\A}))$ for every $x\in\AppInv(\A).$
	\end{corollary}
	Next, we give a sufficient condition for the set of approximately invertible elements in a non-unital abelian Banach algebra with an approximate identity under which its interior is empty.
	
	\begin{corollary}\label{coro-open-int-emp}
		Let $\A$ be a non-unital abelian Banach algebra with an approximate identity. If the Gelfand transform from $\A$ into $C_{0}(\Characters{\A})$ is an open mapping, then the interior of  $\AppInv(\A)$ is empty.
	\end{corollary}
	\begin{proof} Let us denote by $\operatorname{Int}B$ and  $\Gamma(x)=\hat{x}$ the interior of the set $B$ and
		the Gelfand transform from $\A$ into $C_{0}(\Characters{\A})$, respectively. If $\Gamma$ is an open mapping, then  $\Gamma\left(\operatorname{Int}(\AppInv(\A))\right)\subset \operatorname{Int}\left(\Gamma(\AppInv(\A))\right)$. Therefore,  by Corollary \ref{HatAppinv-AppinvCo} we have $\Gamma(\AppInv(\A))\subset \AppInv(C_{0}(\Characters{\A}))$ and hence\linebreak $\operatorname{Int}\left(\Gamma(\AppInv(\A))\right)\subset \operatorname{Int}(\AppInv(C_{0}(\Characters{\A})))$. Now, by Proposition \ref{app-inv-C0T} the proof is completed.
	\end{proof}
	
	\begin{remark}\label{rem-conse-int-empty}
		Let $\A$ be a non-unital abelian Banach algebra with an approximate identity.  If the set $\widehat{\A}=\{\hat{x}\colon x\in\A\}$ is a closed subspace of $C_{0}(\Characters{\A})$ with $x\mapsto\hat{x}$ being the Gelfand transform, then by Open Mapping Theorem such  transform is open and hence by Corollary \ref{coro-open-int-emp} the set $\AppInv(\A)$ has empty interior.
	\end{remark}

	\subsection{Involutive algebras, homomorphisms and representations}
	
	It is very natural to work with abstract algebras by means of homomorphisms with other well-known algebras or by their representations on $\BH$.  
	From now on we consider involutions and norms in $\A$  satisfying \eqref{norminvolution}.
	First we shall prove that the morphisms of normed algebras
	having dense images 
	preserve the bounded approximate right (or left) invertibility.
	
	\begin{proposition}\label{proposition-homomorphism}
		Let $\A, \B$ be normed algebras, and $\pi:\A\rightarrow\B$ be a continuous homomorphism with $\overline{\pi(\A)}=\B$. If $x$ is boundedly approximately right invertible in $\A$,
		then $\pi(x)$ is boundedly  approximately right invertible in $\B$.
	\end{proposition}
	
	\begin{proof}
		The case $\pi=0$ is trivial; we assume that $\pi\ne0$.
		Since $x$ is boundedly approximately right invertible, there is a net $(r_{j})_{j\in J}$ in $\A$  such that $(xr_{j})_{j\in J}$ is a approximate identity in $\A$, which is bounded, say by $M>0$.
		We are going to prove that the net $(\pi(x)\pi(r_j))_{j\in J}$
		is a bounded approximate identity in $\B$.
		The boundedness is obvious: $\|\pi(x)\pi(r_j)\|\le M\|\pi\|$
		for every $j\in J$.
		
		Let $v\in\B$ and $\varepsilon>0$.
		By the assumption $\overline{\pi(\A)}=\B$,
		we can find $u\in\A$ such that
		\begin{equation}\label{e1}
			\|\pi(u)-v\|<\frac{\varepsilon}{2(M\|\pi\|+1)}.
		\end{equation}
		Since $(xr_j)_{j\in J}$ is an approximate identity in $\A$,
		there exists $j_{0}$ such that for every $j\succeq j_{0}$
		\begin{equation}\label{e2}
			\|xr_{j}u-u\|<\frac{\varepsilon}{2\|\pi\|}.
		\end{equation}
		Thus, by \eqref{e1} and \eqref{e2}, for every $j\succeq j_0$ we get 
		\begin{align*}
			\|\pi(x)\pi(r_{j})v-v\|
			&\leq\|\pi(xr_{j})\|\,\|v-\pi(u)\|+\|\pi(xr_{j}u-u)\|+\|\pi(u)-v\|\\
			&\leq\,(M\|\pi\|+1)\|v-\pi(u)\|+\|\pi\|\,\|(xr_{j}u-u)\|<\frac{\varepsilon}{2}+\frac{\varepsilon}{2}=\varepsilon.
		\end{align*}
		It can be proved in a similar manner that the net
		$(v\pi(x)\pi(r_j))_{j\in J}$ converges to $v$.
	\end{proof}
	
	\begin{definition}\cite{Arveson}
		Let $\Ho$ be a complex Hilbert space. Every homomorphism $\pi$ of an $*$-algebra $\A$ into the $C^{*}$-algebra $\BH$ such that $\pi(x^{*})=\pi^{*}(x)$ will be called a \emph{representation} of the algebra $\A$ in the space $\Ho$. If $\overline{\pi(\A)\Ho}=\Ho$, the representation $\pi$ is said to be  \emph{non-degenerate}. 
	\end{definition}

	It is known \cite[Theorem 25.10]{Zelazko} that every representation $x\mapsto\pi(x)$ of an $*$-algebra $\A$ is continuous and satisfies $\|\pi(x)\|\leq\|x\|$. Therefore, we can deduce an analogous result to Proposition~\ref{proposition-homomorphism}.
	
	\begin{proposition}\label{representationandAppInvr}
		Let $\Ho$ be a complex Hilbert space, and $\pi$ be a non-degenerate representation of the $*$-algebra $\A$ in $\Ho$. If $x\in\A$ is boundedly approximately right invertible, then
		\begin{enumerate}
			\item [1.] there is a net $(r_{j})_{j\in J}$ in $\A$ such that $(\pi(xr_{j}))_{j\in J}$ converges in the strong operator topology to the identity operator;
			\item [2.] $\overline{\pi(x)\Ho}=\Ho$.
		\end{enumerate}
	\end{proposition}
	
	\begin{proof}
		1.  Let $x$ be boundedly approximately right invertible. Then there is a net $(r_{j})_{j\in J}$ in $\A$ such that $\sup_{j}\|xr_{j}\|\leq M$ for some $M>0$ and $(xr_{j})_{j\in J}$ is an approximation of the identity in $\A$.
		
		Let  $v\in\Ho$ and $\varepsilon>0$.  Then  we can find $u\in\Ho$ and $a_{v}\in\A$ such that \begin{equation}\label{1piav}
			\|\pi(a_{v})u-v\|<\dfrac{\varepsilon}{2(M+1)}
		\end{equation}
		because the representation is non-degenerate.
		Now, for  $a_{v}\in\A$ and   $\varepsilon>0$, there exists $j_{0}=j_{0}(a_{v},\varepsilon)$  such that 
		\begin{equation}\label{2piav}
			\|xr_{j}a_{v}-a_{v}\|<\dfrac{\varepsilon}{2\|u\|}\quad\text{for all}\ j\succeq j_{0}.
		\end{equation}
		Therefore, by \eqref{1piav} and \eqref{2piav} for all $j\succeq j_{0}$ we have 
		\begin{align*}
			\|\pi(xr_{j})v-v\|
			&\leq\left(\|\pi(xr_{j})\|+1\right) \|v-\pi(a_{v})u\|+\|\pi(xr_{j}a_{v}-a_{v})u\|\\
			&\leq(M+1)\|v-\pi(a_{v})u\|+\|xr_{j}a_{v}-a_{v}\|\,\|u\|<\varepsilon.
		\end{align*}
		Thus, $(\pi(xr_{j}))_{j\in J}$ converges to $\mathrm{Id}_{\Ho}$ in strong operator topology.
		
		2.  Let $(r_{j})_{j\in J}$ be the net from the part (a).
		Hence, given $u\in\Ho$, we define a net $(y_j)_{j\in J}$ 
		in $\pi(x)\Ho$ by $y_j=\pi(r_j)u$
		and obtain
		$\lim_{j\in J}\pi(x)y_{j}=\lim_{j\in J}\pi(x)\pi(r_{j})u=\lim_{j\in J}\pi(xr_{j})u=u.$
		That shows that $\overline{\pi(x)\Ho}=\Ho$.
	\end{proof}
	
	\begin{corollary}
		Let $\A$ be a $*$-algebra and $\pi$ be a non-degenerate
		representation of $\A$ in a complex Hilbert space $\Ho$.
		If $x\in\A$ is boundedly approximately left invertible, then $\pi(x)$ is injective.
	\end{corollary}
	\begin{proof}
		If $x\in\A$ is boundedly approximately left invertible, then by \eqref{norminvolution} and Propositions \ref{AppInvr-AppInvl}, \ref{representationandAppInvr} we get that $x^{*}$ is boundedly approximately right invertible in $\A$, and $\overline{\pi(x^{*})\Ho}=\Ho$. Moreover, 
		$\left[\ker\,\pi(x)\right]^{\perp}=\overline{\operatorname{Rang}\pi^{*}(x)}=\overline{\pi(x^{*})\Ho}=\Ho,$
		so $\ker\,\pi(x)=\{0\}$. Thus, $\pi(x)$ is injective.
	\end{proof}

	\section{Approximate invertibility in examples}\label{Sec:Examp}
	
	In this section we provide several examples of algebras with as well as without approximatively invertible elements. These algebras were a motivation for investigating approximate invertibility in some classes of algebras studied in Section~\ref{sec:AppInv in algebras}.
	
	\subsection{Small disk algebra}\label{Section:SmallDiskAlgebra}

	Denote by $\DiskAlgebra$ the usual disk algebra
	consisting of all continuous functions
	$f\colon\overline{\mathbb{D}}\to\mathbb{C}$
	which are analytic in $\mathbb{D}$.
	Consider the ``small disk algebra'' defined as
	\[
	\SmallDiskAlgebra\eqdef \{f\in\DiskAlgebra\colon\ f(0)=0\}.
	\]
	
	Since the character space of $\DiskAlgebra$
	can be naturally identified with $\overline{\mathbb{D}}$
	and $\SmallDiskAlgebra$ is the maximal ideal of $\DiskAlgebra$
	associated to the point $0$,
	we can conclude by~\cite[Theorem 7.3.1]{Larsen}
	that $\M_{\SmallDiskAlgebra}$ can be naturally identified with
	$\overline{\Disk}\setminus\{0\}$.
	
	Notice that $\SmallDiskAlgebra$ is generated by the monomial function
	$\monomial\colon\overline{\mathbb{D}}\to\mathbb{C}$ such that
	$\monomial(z)=z$. 
	By Schwarz lemma, for every $f\in\A_0$ and every $z\in\mathbb{D}$ it holds
	$|f(z)|\le |z|\,\|f\|_\infty.
	$ The following lemma shows that the elements of $\A_0$
	can not be uniformly close
	to the constant $1$ in the annulus $1/2\le|z|\le 1$.
	
	\begin{lemma}\label{lem:small_disk_algebra_far_from_unity}
		If $f\in\A_0$, then 
		\begin{equation}\label{eq:small_disk_algebra_far_from_unity}
			\sup_{1/2\le|z|\le1}|f(z)-1|\ge\frac{1}{3}.
		\end{equation}
	\end{lemma}
	
	\begin{proof}
		Denote $\|f\|_\infty$ by $M$. Note that
		$M=\sup_{|z|=1}|f(z)|.$ If $M\ge 4/3$, then $\displaystyle
		\sup_{1/2\le|z|\le1}\hspn|f(z)-1|
		\ge\hspn M-1\hspn \ge\frac{1}{3}.$
		If $M<4/3$, then by Schwarz lemma
		$|f(1/2)|\le M/2 < 2/3$,
		and $\displaystyle
		\sup_{1/2\le|z|\le1}|f(z)-1|
		\ge 1-|f(1/2)|\ge \frac{1}{3}.$
		In both cases the inequality \eqref{eq:small_disk_algebra_far_from_unity} holds.
	\end{proof}
	
	In the next proposition we show that $\chi_{1}$
	can not be approximated by products
	of two (or more) elements of $\SmallDiskAlgebra$.
	
	\begin{lemma}\label{lem:monomial_can_not_be_approximated}
		For every $f_1,f_2\in\SmallDiskAlgebra$ it holds
		$\|f_1 f_2 - \monomial\|_\infty \ge \frac{1}{3}.
		$
	\end{lemma}
	
	\begin{proof}
		Since $\SmallDiskAlgebra$ is generated by $\monomial$,
		we can write $f_1(z)f_2(z)$ as $z^2 h(z)$
		with some function $h$ analytic on $\mathbb{D}$
		and continuous on $\overline{\mathbb{D}}$.
		Therefore, by Lemma \ref{lem:small_disk_algebra_far_from_unity} we get
		\begin{align*}
			\|f_1 f_2 - \monomial\|_\infty&=\sup_{|z|=1}\left|z^{2}h(z)-z\right|=\sup_{z\in \overline{\mathbb{D}}}\left|zh(z)-1\right|\geq \sup_{1/2\leq|z|\leq1}\left|zh(z)-1\right|\geq\frac{1}{3} 
		\end{align*}providing the result.
	\end{proof}
	
	As a corollary
	of Lemma~\ref{lem:monomial_can_not_be_approximated} we immediately have the main properties of the small disk algebra.
	
	\begin{proposition}\label{prop:A0_is_poor}
		The algebra $\SmallDiskAlgebra$ has no approximate identities,
		no approximately invertible elements, and no dense principal ideals.
	\end{proposition}
	
	The function $\monomial$ does not vanish in any point of $\overline{\Disk}\setminus\{0\}$,
	but the principal ideal $\monomial \SmallDiskAlgebra$ is not dense.
	Thus, $\SmallDiskAlgebra$ provides a counterexample to converse implication in Proposition~\ref{prop:ainv_modid}. Moreover, $\monomial$ is not a topological divisor of zero in $\SmallDiskAlgebra$.
	In fact, $\|f\monomial\|_\infty=\|f\|_\infty$ for every $f\in \SmallDiskAlgebra$. So, in the algebra $\SmallDiskAlgebra$
	there are no approximately invertible elements, but not every element is a topological divisor of zero.

	\subsection{Wiener algebras}

	For a locally compact Hausdorff space $X$, a normed algebra $(W,\|\cdot\|_{W})$ of complex-valued continuous functions on $X$ with pointwise operations is said to be a \textit{Wiener algebra}
	(see Reiter and Stegeman \cite[Chapter 2]{ReiterStegeman}), if
	
	\begin{enumerate}[label=(WA\arabic*),
		labelsep=*,
		leftmargin=*,
		widest=(WA1)]
		
		\item for each $t\in X$,  the evaluation  functional $f\mapsto f(t)$ on $W$ is continuous;
		\label{WA1}
		
		\item for any closed set $E\subsetneq X$ and any point $t\notin E$, there is a function $f\in\,W$ vanishing on $E$ and such that $f(t)\neq0$;
		\label{WA2}
		
		\item if $f\in W$ with $f(t)\neq0$ at a point $t\in X$, then there is a function $g\in\,W$ such that $g(x)=\dfrac{1}{f(x)}$ for all $x$ in some neighborhood of $t$;
		\label{WA3}
		
		\item the compactly supported elements in $W$ form a dense subspace.
		\label{WA4}
		
	\end{enumerate}
	
	In this section, we suppose that $X$ is a locally compact non-compact Hausdorff space,
	$(W,\|\cdot\|_{W})$ is a Wiener algebra on $X$, and, additionally,
	that $W$ has an approximate identity. Due to condition~\ref{WA2}, for each point $t\in X$ there exists a function $f\in W$ such that $f(t)=1$; thus, the kernel of every evaluation functional is a maximal modular ideal. On the other hand, it is known~\cite[Proposition 2.6.1]{ReiterStegeman}
	that maximal closed ideals in $W$ correspond to the evaluation functionals. Therefore,  $\Characters{W}$ can be naturally identified with $X$.

	\begin{theorem}
		\label{thm:Wiener_ainv-Wie-alg}
		Let $f\in W$. Then the following conditions are equivalent:
		\begin{enumerate}
			\item  There exists is a net $(h_j)_{j\in J}$ in $W$
			such that $(f h_j)_{j\in J}$ is an approximate identity in $W$,
			and the functions $fh_j$ have compact supports.
			\item  $f$ is approximately invertible in $W$.
			\item  $fW$ is dense in $W$.
			\item The Gelfand transform $\widehat{f}$ of $f$
			does not vanish on $\Characters{W}$.
			\item $f$ does not vanish on $X$.
		\end{enumerate}
	\end{theorem}
	
	\begin{proof}
		We only prove the implication (v)$\Rightarrow$(i),
		since the remaining implications are simple.
		By assumption, $W$ has an approximate identity, hence by~\ref{WA4} and by~\cite[Proposition 2.4.2]{ReiterStegeman} we have that $W$ has an approximate identity $(\aid_{j})_{j\in J}$ where $\aid_{j}$ has compact support for every $j$ in $J$. Suppose that $f(t)\neq0$ for every $t\in X$. Then for every $j$ in $J$,
		by the analogue of Wiener's division lemma for Wiener algebras
		~\cite[Corollary 2.1.11]{ReiterStegeman},
		there exists $g_{j}$ in $W$ such that $f g_{j}=\aid_{j}$.
	\end{proof}

	The following consequence is similar to Proposition~\ref{app-inv-C0T} and Corollary~\ref{coro-open-int-emp}.
	
	\begin{corollary}
		Let $(W,\|\cdot\|_{W})$ be a Wiener algebra with an approximate identity. Then the interior of $ \AppInv(W)$ is empty.
	\end{corollary}
	
	\begin{proof}
		By Theorem~\ref{thm:Wiener_ainv-Wie-alg}, the functions  with compact support in $X$ are not approximately invertible. Thus,  any ball $B_{\delta}(f)$   centered at $f\in W$ with radius $\delta$ contains functions  that are not approximately invertible and hence the interior of $ \AppInv(W)$ is empty.
	\end{proof}

	\begin{example}[the Fourier algebra]
		\label{example:convolution_algebra}
		Let $G$ be a non-discrete locally compact abelian group
		equipped with a Haar measure $\mu$.
		We shall use the additive notation for $G$.
		The space $\Lone(G)$ with the convolution operation $\ast$
		is a non-unital commutative algebra.
		We identify the space $\Characters{\Lone(G)}$
		with the dual group $\widehat{G}$.
		Then the Gelfand transform for the algebra $\Lone(G)$
		coincides with the Fourier transform
		$\mathcal{F}\colon\Lone(G)\to C_0(\widehat{G})$,
		and the image of this transform, i.e., $\mathcal{F}(\Lone(G))$,
		is known as the Fourier algebra
		(some authors call it the Wiener algebra).
		Of course, this is a typical example of a Wiener  algebra, and the previous results of this section can be applied to this example.
		It is well known that $\mathcal{F}(\Lone(G))$
		has norm bounded approximate identities
		(with norms bounded by $1$)
		and that the functions in
		$\mathcal{F}(\Lone(G))$ with compact
		support are dense there
		(see \cite[Theorem 8.1.2, pp. 184--187]{Larsen}
		or \cite[Proposition 5.4.1]{ReiterStegeman}).
		Therefore, by \cite[Lemma 1.4]{Doran-Wichmann},
		$\mathcal{F}(\Lone(G))$
		contains an approximate  identity $(\aid_{j})_{j\in J}$
		such that the functions $\aid_{j}$ have compact supports for each $j$ in $J$.
		For this example,
		the implication (v)$\Rightarrow$(i)
		from Theorem~\ref{thm:Wiener_ainv-Wie-alg}
		can be proved
		by the classical Wiener's Division Lemma~\cite[Lemma 1.4.2]{ReiterStegeman},
		and the implication (v)$\Rightarrow$(iii)
		can also be proved in a different way,
		by~\cite[Theorem 6.1.4]{ReiterStegeman}
		and Proposition~\ref{prop:rich_modular_ideals}.
	\end{example}
	
	\begin{example}[Segal algebras]
		\label{example:Fourier-Segal-algebra}
		Let $G$ be a non-discrete locally compact abelian group
		equipped with a Haar measure $\mu$. As in Example \ref{example:convolution_algebra}, we shall use the additive notation for $G$.
		A subalgebra $S$ of $\Lone(G)$ is said to be a \emph{Segal algebra} if 
		
		\begin{enumerate}[label=(SA\arabic*),
			labelsep=*,
			leftmargin=*,
			widest=(SA1)]
			
			\item
			\label{SA1}
			the space $S$ is dense in $\Lone(G)$ with respect to the norm $\|\cdot\|_{1}$;
			
			\item
			\label{SA2}
			$S$ is invariant under translations, i.e., $L_y f\in S$ for every $f$ in $S$ and $y$ in $G$,
			where $L_{y}f(x)=f(x-y)$;
			
			\item
			\label{SA3}
			$S$ is a Banach algebra under some norm $\|\cdot\|_{S}$ which is invariant under translations;
			
			\item
			\label{SA4}
			for every $f\in S$ and every $\eps>0$,
			there is a neighborhood $U$ of $0$ such that
			$\|L_{y}f-f\|_{S}<\eps$ for all $y$ in $U$.
		\end{enumerate}
		
		Let $W$ be the the Fourier image of $S$
		(i.e., $\mathcal{F} S$),
		considered with the norm carried over from $S$.
		Then $W$ is a subalgebra of $\mathcal{F}(\Lone(G))$,
		and $W$ is isomorphic to $S$.
		Moreover, $W$ is a Wiener algebra and possesses approximate identities,
		see~\cite[Proposition 6.2.8]{ReiterStegeman}.
		If $S$ is a non-trivial Segal algebras, i.e., $S\subsetneq \Lone(G)$, then  the approximate identities in $S$ are  bounded with respect to the $\Lone$-norm but they are never bounded  in the norm $\|\cdot\|_{S}$, see e.g. \cite[Theorem 1.2]{Burnham}. Now,  \cite[Proposition 6.2.6]{ReiterStegeman} implies that elements of $W$ have bounded approximate units, and by 
		\cite[Proposition 9.5]{Doran-Wichmann}
		this implies that $W$ has an approximate identity (possibly unbounded).
	\end{example}
	
	\begin{example}[Beurling algebras]
		This generalization of Example~\ref{example:convolution_algebra}
		was studied by Domar~\cite{Domar1956}.
		Let $G$ be a non-discrete locally compact abelian group equipped with a Haar measure $\mu$,
		and let $w\colon G\to[1,+\infty)$
		be a measurable function with respect to the Haar measure,
		such that $w$ is bounded on every compact set and submultiplicative in the following sense:
		\[
		w(x+y)\le w(x)w(y)\qquad(x,y\in G).
		\]
		Moreover, we suppose that $w$ satisfies the Beurling--Domar condition:
		\[
		\sum_{n\geq1}\dfrac{w(nx)}{n^{2}}<+\infty\qquad (x\in G).
		\]
		Then $\Lone(G, w\,\dif\mu)$,
		considered with the norm $\|\cdot\|_{1,w}$
		and with the convolution operation,
		is a Banach algebra,
		and its Fourier image 
		$W\eqdef\cF \Lone(G,w\,\dif\mu)$
		is a Wiener algebra, see \cite[Proposition 6.3.2 and Lemma~A.1.4]{ReiterStegeman} based on various results from~\cite{Domar1956}.
		Moreover, \cite[Proposition 3.7.6]{ReiterStegeman} implies that elements of $W$ have bounded approximate units, and by 
		\cite[Proposition 9.5]{Doran-Wichmann}
		this implies that $W$ has an approximate identity (possibly unbounded).
	\end{example}

	\subsection{Operator ideals}
	
	In this section we exemplify the approximate invertibility in certain operator algebras. First we recall a few known facts for operators, see for instance \cite[\textsection2.4]{Brezis}, providing a motivation for the detailed study.  A linear operator $T\in\mathcal{B}(\Ho)$ is right invertible if there exists $S\in\mathcal{B}(\Ho)$ such that $TS=\operatorname{Id}$. The existence of a right inverse of  $T\in\mathcal{B}(\Ho)$ is guaranteed if and only if $T$ is surjective. On the other hand, a linear operator $T\in\mathcal{B}(\Ho)$ is left invertible if there exists $S\in\mathcal{B}(\Ho)$ such that $ST=\operatorname{Id}$, and the existence of left inverse of  $T\in\mathcal{B}(\Ho)$ is guaranteed if and only if $T$ is injective  and $\Rang(T)$ is closed (or $T$ is bounded below). These results can be extended to linear operators acting on Banach spaces, but the concept of complemented subspaces is needed, see for instance \cite[Theorem 2.12, Theorem 2.13]{Brezis}.

	In what follows, $\Ho$ is an infinite-dimensional separable Hilbert space.
	Then the identity operator in $\Ho$ is not compact,
	and the compact operators acting on $\Ho$ cannot be neither \textit{bounded below} nor \textit{surjective},
	see~\cite[Theorem 4.18]{Rudin}.
	Thus, compact operators acting on $\Ho$
	are neither right, nor left invertible.
	Thence, finding conditions for a  weaker form of invertibility of compact operators is an important task. We will address this task in a slightly more general setting of operator ideals.
	
	We denote by $\mathfrak{F}(\Ho)$
	the collection of all  operators $F$ in $\BH$ such that $\operatorname{Rank}(F)\eqdef \dim\Rang(F)$ is finite. The elements of $\mathfrak{F}(\Ho)$ are finite linear combinations of $f\otimes g$, where $(f\otimes g)(h)\eqdef \langle h,g\rangle f$.
	
	For $n$ in $\N$, we refer to 
	$
	a_{n}(S)\hspn=\hspn\inf\left\{\,\|S-F\|\colon \hspn\operatorname{Rank}(F)\hspn<n\right\}
	$
	as $n$-th approximation number of $S\in\BH$.
	The C*-algebra $\KH$ of compact operators
	(also known as \emph{completely continuous} operators), acting on $\Ho$, may be characterized as
	$\displaystyle
	\CompactOperators(\Ho)=\left\{S\in\BH\colon\lim_{n\rightarrow\infty}a_{n}(S)=0\right\}.$

	The study of operador ideals was started by Calkin~\cite{Calkin-1941}.
	Following~\cite{Pietsch-2017},
	we say that a subspace $\mathfrak{U}$ of $\BH$
	is an \emph{operator ideal}
	if for any $S\in\BH$ and $T\in\mathfrak{U}$ such that  $a_{n}(S)\leq a_{n}(T)$,  we have $S\in\mathfrak{U}$.
	
	Obviously, this definition implies that $\fU$ is a two-sided ideal.

	In this section, we fix a proper operator ideal $\fU$ on $\Ho$, $\fF(\Ho)\subseteq  \fU\subseteq\CompactOperators(\Ho)$.
	We assume that:
	
	\begin{enumerate}[label=(OI\arabic*),
		labelsep=*,
		leftmargin=*,
		widest=(OI1)]
		
		\item
		\label{OI1}
		$\fU$ is provided with a norm $\|\cdot\|_{\fU}$,
		
		\item
		\label{OI2}
		$\fU$ is complete with respect to the induced uniformity,
		
		\item
		\label{OI3}
		$\fF(\Ho)$ is dense in $\fU$
		with respect to $\|\cdot\|_{\fU}$, and
		
		\item
		\label{OI4}
		the normalization condition
		$\|f\otimes g\|_{\fU}=\|f\|\,\|g\|$
		holds for every $f,g$ in $\Ho$.
	\end{enumerate}
	
	By \cite[Proposition 6.1.4]{Pietsch-1980},
	it follows that 
	$\|A\|\le \|A\|_{\fU}$ for every $A$ in $\fU$. Notice that $\|\cdot\|_{\fU}$ is equivalent to the operator norm $\|\cdot\|$ only if $\fU=\CompactOperators(\Ho)$.
	
	The most important particular cases are the Schatten classes, including the ideal of all compact operators. More generally, Pietsch~\cite[\textsection 6.1]{Pietsch-1980} considered operator ideas provided with quasi-norms, and the results of this section can be extended to that context (treating $\fU$ as a topological algebra).
	
	A fundamental tool for studying the operator ideals in $\BH$ is the singular value decomposition theorem, also known as Schmidt's Theorem (see, for example, \cite[Theorem D.3.2]{Pietsch-1980}).
	According to this theorem, each operator $S\in\CompactOperators(\Ho)$ admits a representation 
	\begin{equation}\label{eq:HSch-opi}
		S\hspace{-2pt}=\sum_{n\in\N}\lambda_{n}  u_{n}\otimes e_{n},
	\end{equation}
	where $\lambda_{n}=a_{n}(S)$, $n\in\N$,
	is a non-increasing sequence converging to zero, 
	and $(u_{n})_{n\in\N}$ and $(e_{n})_{n\in\N}$ are orthonormal sequences in $\Ho$.
	
	First, we characterize the approximate identities in $\fU$ that are bounded in the uniform norm.

	\begin{proposition}\label{prop-sec-appr-id-comp-opi2}
		Let $(S_{\hspn j})_{j\in J}$ be a net in $\fU$ such that $\sup_{j\in J}\|S_{\hspn j}\|<\infty$. Then the following assertions are equivalent: \begin{enumerate}
			\item  $(S_{\hspn j})_{j\in J}$ is an approximate identity in $\fU$; 
			\item  $S_{\hspn j} v\to v$ and $S_{\hspn j}^{*} v\to v$ for every $v\in \Ho$.
		\end{enumerate}
	\end{proposition}
	
	\begin{proof}
		(i)$\Rightarrow$(ii).
		Suppose $(S_{\hspn j})_{j\in J}$ is an approximate identity in $\mathfrak{U}$.
		Let $v\in\Ho$.
		Denote by $P_v$ the orthogonal projection in $\Ho$ such that $P_v(\Ho)=\linspan\{v\}$.
		Then $P_v\in\fF(\Ho)\subseteq\fU$,
		hence the nets $(S_{\hspn j} P_v)_{j\in J}$ and $( P_vS_{\hspn j})_{j\in J}$ converge to $P_v$ in $\fU$.
		Since
		\begin{align*}
			\|S_{\hspn j}v-v\|
			&=\|S_{\hspn j}P_{v}v-P_{v}v\|
			\le \|S_{\hspn j}P_{v}-P_{v}\|\,\|v\|
			\le  \|S_{\hspn j}P_{v}-P_{v}\|_{\fU}\,\|v\|,\\
			\|S_{\hspn j}^{*}v-v\|
			&=\|S_{\hspn j}^{*}P_{v}v-P_{v}v\|
			\le \|S_{\hspn j}^{*}P_{v}-P_{v}\|\,\|v\|=\|P_{v}S_{\hspn j}-P_{v}\|\,\|v\|
			\le  \|P_{v}S_{\hspn j}-P_{v}\|_{\fU}\,\|v\|,
		\end{align*}
		we conclude that $S_{\hspn j} v\to v$ and $S_{\hspn j}^{*} v\to v$.
		
		(ii)$\Rightarrow$(i).
		Let $(S_{\hspn j})_{j\in J}$ be such that $\displaystyle\sup_{j\in J}\|S_{\hspn j}\|<\infty$ and $S_{\hspn j}v\rightarrow v$ for every $v\in \Ho$. Let us show that $\lim_{j\in J}\|S_{\hspn j} T - T\|_{\fU}=0$,  when $T$ is of the form $T\eqdef u\otimes e$, with $u,e\in\Ho$. Then 
		$S_{\hspn j} T =
		(S_{\hspn j} u)\otimes e$ and
		\[
		\|S_{\hspn j} T - T\|_{\fU}
		=\|(S_{\hspn j} u) \otimes e - u \otimes e\|_{\fU}
		=\|(S_{\hspn j} u - u)\otimes e\|_{\fU}
		=\|S_{\hspn j} u - u\|\,\|e\|,
		\]
		and the last expression tends to zero.
		
		If $T\in\fF(\Ho)$, then $T$ is a finite linear combination of operators $ u\otimes e$, and $\displaystyle\lim_{j}\|S_{\hspn j} T - T\|_{\fU}=0$. 
		
		Let  $C\in\mathfrak{U}$ and $\varepsilon>0$.
		Then,  there exists a finite-rank linear operator $T_{\hspn N}$ such that $\|T_{\hspn N}-C\|_{\fU}\hspn< \frac{\varepsilon}{2(M+1)} $, where $M=\sup_{j\in J}\|S_{\hspn j}\|$, because $\fF(\Ho)$ is dense in $\fU$ with respect to  $\|\cdot\|_{\fU}$. By the above argument there exists $j_{0}\in J$ such that $\|S_{\hspn j} T_{\hspn N}-T_{\hspn N}\|_{\fU}< \dfrac{\varepsilon}{2}$ for all $j\succeq j_{0}$. In this way, we have  for all $j\succeq j_{0}$
		\begin{align*}\label{aux-n-3-c-opi2}
			\|S_{\hspn j}C-C\|_{\fU}
			&\leq \|T_{\hspn N}-C\|_{\fU}
			+\|S_{\hspn j}C-S_{\hspn j}T_{\hspn N}\|_{\fU}
			+\|S_{\hspn j}T_{\hspn N}-T_{\hspn N}\|_{\fU}\nonumber\\
			&\leq
			\left(\sup_{j\in J} \|S_{\hspn j}\|+1\right) \|C-T_{\hspn N}\|_{\fU}
			+\|S_{\hspn j}T_{\hspn N}-T_{\hspn N}\|_{\fU}
			< \varepsilon.
		\end{align*} 
		The proof that $(TS_{\hspn j})_{j\in J}$  converges to $T\in\fU$ follows almost literally as above, using the identity $(u\otimes e)S_{\hspn j}=u\otimes (S_{\hspn j}^{*}e)$.  Consequently,  $(S_{\hspn j})_{j\in J}$ is an approximate identity in $\mathfrak{U}$.
	\end{proof}

	\begin{proposition}\label{prop:projections_sequence-opi2}
		Let $(b_j)_{j\in \N}$
		be an orthonormal basis for $\Ho$.
		For every $m$ denote by $S_{\hspn m}$ the orthogonal projection
		onto the subspace generated by $b_1,\ldots,b_m$:
		\[
		S_{\hspn m} = \sum_{j=1}^m b_{j}\otimes b_{j}, 
		\]
		Then $(S_{\hspn m})_{m\in\N}$
		is an approximate identity for $\mathfrak{U}$.
	\end{proposition}
	
	\begin{proof}
		Since $S_{\hspn m}\in\mathfrak{U}$, $S_{\hspn m}=S_{\hspn m}^{*}$,  $\|S_{\hspn m}\|=1$, for every $m\in\N$, and $S_{\hspn m}v\rightarrow v$ for every $v\in \Ho$, by Proposition \ref{prop-sec-appr-id-comp-opi2} we conclude that $(S_{\hspn m})_{m\in\N}$ is an approximate identity in $\fU$.
	\end{proof}
	
	Now, we are ready to provide necessary and sufficient conditions for the left and right invertibility in operator ideals.
	
	\begin{theorem}\label{prop:ApprInvCompact-opi}
		Let $T\in\mathfrak{U}$. Then following statements hold:
		\begin{enumerate}
			\item[1.] $T\in\AppInv_{r}(\fU)$ 
			if and only if  $T(\Ho)$ is dense in $\Ho$.
			\item[2.] $T\in\AppInv_{l}(\fU)$ 
			if and only if $\ker(T)=\{0\}$.
		\end{enumerate}
	\end{theorem}
	
	\begin{proof}
		1. Suppose that $T(\Ho)$ is not dense in $\Ho$. Choose $a\in\Ho$ such that $\|a\|=1$ and $a\perp T(\Ho)$. Then the orthogonal projection $P_{a}=a\otimes a$ from $\Ho$ onto $\operatorname{span}\{a\}$ belongs to $\mathfrak{U}$ and  for every operator $U\in\mathfrak{U}$ we have
		$\|TU-P_a\|_{\fU}\ge\|TU-P_a\|\ge \|TUa-P_a a\|\ge\operatorname{dist}(a,T(\Ho))=1.$
		Since the right ideal $T\mathfrak{U}$
		is not dense in $\mathfrak{U}$, the operator $T$ is not approximately right invertible  in $\mathfrak{U}$ by Theorem \ref{thm:dense_ideals_implies_ainv}.
		
		Let $T\in\mathfrak{U}$ be such that $T(\Ho)$ is dense in $\Ho$. Write $T$ in the form \eqref{eq:HSch-opi} with $(u_{n})_{n\in\N}$ being an orthonormal basis of $\Ho$
		and $\lambda_j=a_{j}(T)>0$ for every $j$.
		Let 
		$\displaystyle U_m=\sum_{k=1}^m \frac{1}{\lambda_k}\,u_{k}\otimes e_{k}$.
		Then
		\[
		T U_m v
		=\sum_{j\in\N} \sum_{k=1}^m \lambda_j\,\frac{1}{\lambda_k}
		\langle v,e_k\rangle\, \langle u_k,u_j\rangle\,e_j
		=\sum_{j=1}^m e_{j}\otimes e_{j}(v).
		\]
		Thus, 
		the sequence $(T U_m)_{m\in\N}$ is an approximate identity
		in $\mathfrak{U}$ by Proposition~\ref{prop:projections_sequence-opi2}.
		
		2.  The proof follows  by  Proposition \ref{AppInvr-AppInvl} and by identity $\ker T=(\Rang(T^{*}))^{\perp}$.
	\end{proof}

	By \cite[Example 5.1.1]{Murphy},
	the pure states of the algebra $\CompactOperators(\Ho)$
	are of the form $\tau_a$ for some $a\in\Ho$ with $\|a\|=1$, where
	$\tau_a(T)\eqdef\langle T a,a\rangle$.
	This fact and the correspondence between the pure states and maximal modular left or right ideals~\cite[Example 5.1.1]{Murphy},
	yields the upcoming simple description of such ideals in the algebra $\CompactOperators(\Ho)$.
	Recall that $N_\tau$ is defined
	by~\eqref{Nt} for every pure state $\tau$.
	In our situation,
	\[
	N_{\tau_a}
	=\{T\in\CompactOperators(\Ho)\colon\ \tau_a(T^\ast T)=0\}
	=\{T\in\CompactOperators(\Ho)\colon\
	\langle T^\ast T a,a\rangle = 0\}
	=\{T\in\CompactOperators(\Ho)\colon\ Ta=0\}.
	\]
	
	\begin{proposition}\label{prop:maximal_modular_ideals_in_CompactOperators}
		1. The maximal right modular ideals of $\CompactOperators(\Ho)$ are of the form $N_{\tau_a}^{\ast}$
		for some $a\in\Ho$ with $\|a\|=1$.
		Moreover,
		\[
		N_{\tau_a}^{\ast}
		=\{T\in\CompactOperators(\Ho)\colon\ 
		\Rang(T)\perp a\}
		=(I-P_a)\CompactOperators(\Ho).
		\]
		2. The maximal left modular ideals of $\CompactOperators(\Ho)$ are of the form $N_{\tau_a}$
		for some $a\in\Ho$ with $\|a\|=1$.
		Moreover,
		\[
		N_{\tau_a}
		=\{T\in\CompactOperators(\Ho)\colon\ 
		a\in\ker(T)\}
		=\CompactOperators(\Ho)(I-P_a).
		\]
	\end{proposition}
	
	In the C*-algebra $\KH$  of compact operators,
	we may apply Theorem~\ref{Appinvl-modi-C*}
	and Proposition~\ref{prop:maximal_modular_ideals_in_CompactOperators}
	for proving
	Theorem~\ref{prop:ApprInvCompact-opi}.

	\section{An application to the density in Banach modules}\label{sec:applications}

	In this section we work with products $ab$, where $a$ and $b$ belong to two different spaces. We suppose that $(\A,\|\cdot\|)$
	is a non-unital Banach algebra and $(B,\|\cdot\|_{B})$ is a left Banach $\A$-module with respect to an operation $\bullet\colon\A\times B\to B$,
	see~\cite[Definition~32.14]{Hewitt-Ross-II}.
	In all this section we will require a strong additional assumption:
	$B=\A\bullet B$.
	
	The following result and its corresponding proof may be found in~\cite[Remark 32.33 (a)]{Hewitt-Ross-II}, we reproduce it here for the sake of completeness.
	
	\begin{proposition}
		\label{prop-module-1}
		Let $(\aid_{j})_{j\in J}$ be a left approximate identity in $\A$.
		Then $(\aid_{j}\bullet b)_{j\in J}$ converges to $b$ for any $b\in B$.
	\end{proposition}
	
	\begin{proof}
		Let $b\in B$. By the assumption $B=\A\bullet B$, there exist $a\in\A$ and $c\in B$ such that $b=a\bullet c$.
		Thus, by~\cite[Definition 32.14]{Hewitt-Ross-II},
		\begin{align*}
			\|\aid_{j}\bullet b-b\|_{B}
			&=\|\aid_{j}\bullet (a\bullet c)-a\bullet c\|_{B}
			=\|(\aid_{j}a)\bullet c-a\bullet c\|_{B}
			\\
			&=\|(\aid_{j}a-a)\bullet c\|_{B}
			\leq \text{const}\,
			\|c\|_{B}\,\|\aid_{j}a-a\|.
		\end{align*}
		The last factor tends to zero,
		because $(\aid_{j})_{j\in J}$
		is a left approximate identity in $\A$.
	\end{proof}
	
	\begin{proposition}
		\label{prop-module-2}
		Let $x\in\AppInv_{r}(\A)$. Then $x\bullet B$ is dense in $B$.
	\end{proposition}
	
	\begin{proof}
		Since $x\in\AppInv_{r}(\A)$, there exists a net $(r_{j})_{j\in J}$ in $\A$ such that $(xr_{j})_{j\in J}$ is an approximate identity in $\A$. Now, given any $b\in B$,
		by Proposition~\ref{prop-module-1} we have that $\bigl((xr_{j})\bullet b\bigr)_{j\in J}$ converges to $b$.
		On the other hand, for every $j$ in $J$ we have
		$(xr_{j})\bullet b
		=x\bullet(r_{j}\bullet b)
		\in x\bullet B$,
		i.e., the net $\bigl((xr_{j})\bullet b\bigr)_{j\in J}$
		takes values in $x\bullet B$.
		Consequently, $\overline{x\bullet B}=B$.
	\end{proof}
	
	In particular, the results can be applied to $\A$ being the convolution algebra $\Lone(G)$ and $B$ being a homogeneous Banach space. 
	Recall that a homogeneous Banach space on a non-discrete locally compact abelian group $G$ is a Banach space $(B,\|\cdot\|_{B})$ of measurable functions on $G$, such that is closed under translations, i.e., given any $f\in B$ we have $L_{x}f\in B$ for all $x\in G$, 
	the norm $\|\cdot\|_{B}$ is invariant under $L_{x}$,  and the map $x\mapsto L_{x}f$ of $G$ into $B$ is continuous. 
	Following~\cite{Wa77}, we define
	$\circledast\colon\Lone(G)\times B\to B$ by
	\[
	f\circledast g=\int_{G}f(x)L_{x}g\,\dif\mu(x),
	\]
	where the integral is understood in the weak sense. It is known~\cite[Theorem 2.11]{Wa77}
	that $B$ is a $\Lone(G)$-module
	and $\Lone(G)\circledast B=B$.

	The following results are particular cases of Propositions~\ref{prop-module-1} and~\ref{prop-module-2}, therefore we omit the proof. Their particular cases are well known, when $(\aid_j)_{j\in J}$ is a Dirac net or a summability kernel in the sense of~\cite{ka76}.

	\begin{corollary} \label{cor:net+dense} Let $(B,\|\cdot\|_{B})$ be a homogeneous Banach space on $G$. 
		\begin{itemize}
			\item[(i)] If $g\in B$ and $(\aid_j)_{j\in J}$ is an approximate identity in $\Lone(G)$, then the net $( g\circledast\aid_j)_{j\in J}$ converges to $g$ in $B$. \item[(ii)] If $f\in\Lone(G)$ is such that $\widehat{f}(t)\ne0$ for every $t\in\widehat{G}$, then the set $f\circledast B$ is dense in $B$.
		\end{itemize}
	\end{corollary}
	
	\begin{example}
		Let $B=\Lp(G)$ or $B=C_{bu}(G)$. Then $f\circledast g$ is just $f\ast g$, see \cite[Remark 2.14]{Wa77}.
		For these spaces, Corollary \ref{cor:net+dense} provides a sufficient condition  for density of  the image of the convolution operator. By Theorem~\ref{thm:Wiener_ainv-Wie-alg},  this condition is necessary in the case $B=\Lone(G)$. 
		A natural problem is to find  necessary and sufficient conditions for other spaces $B$. \end{example}

	Above results are in close connection with recent papers of the second author~\cite{fegu20,fegu21} addressing the question of completeness of sets of shifts in minimal tempered standard spaces. In this situation non-vanishing Fourier transform implies that the set of translates is total in the Banach algebra.
	
	\begin{example}
		Let $(S,\|\cdot\|_{S})$ be a Segal algebra
		in $\Lone(G)$, see Example \ref{example:Fourier-Segal-algebra}. This means that $S$ is a homogeneous Banach space  and a dense subspace of $\Lone(G)$.  
		By \cite{du74} or by \cite[Theorem 2.16]{Wa77},  we have  $S=\Lone(G)\ast S$. So, Corollary \ref{cor:net+dense} 
		can be applied to $B=S$ with $f\circledast g=f\ast g$.
	\end{example}

	\section{Several questions and open problems}\label{Sec:rem-open-pro}
	
	With the development of this project we have seen that approximate invertibility is a generalization of invertibility and is related with the density of principal ideals.  
	Now we present some remarks, questions, and open problems related to the new concept, its properties, characterizations or extensions. 
	
	An immediate observation is that given a non-unital normed algebra with an approximate identity and non-trivial idempotent (or nilpotent) elements, the set of all approximately right invertible elements is proper. Indeed, if $\A$ is a non-unital normed algebra with an approximate identity and $x$ is a non-trivial idempotent or nilpotent element in $\A$, i.e., $x^{2}=x$ or $x^{n}=0$, for some $n\geq1$, respectively, then $x$ does not belong to $\AppInv_{r}(\A)$. Otherwise, there exists a net $(r_{j})_{j}$ in $\A$ such that $(xr_{j})_{j}$ is an approximate identity for $\A$. However, if $x=x^{2}\neq0$, then
	$\lim_{j}x(xr_{j})=\lim_{j}xr_{j}=x.$
	Therefore, the net $(xr_{j})_{j}$ has limit $x$ and hence $\A$ is unital by Proposition~\ref{prop:convergentAI}, which is a contradiction.  
	
	\begin{question}
		Are there non-unital normed algebras with approximate identities and non-trivial idempotent elements such that  this idempotent element does not belong to any maximal modular right ideal?
	\end{question}
	
	Despite several characterizations of approximately invertible elements in some classes of algebras presented in the paper, we were not able to find concrete algebras violating certain properties. One such case is described in the following problem.
	
	\begin{question}
		Are there non-unital Banach algebras with approximate identities such that the conditions (i) and (iii) from Proposition \ref{prop:rich_modular_ideals} are not equivalent?
	\end{question}

	In this paper we have mainly dealt with normed algebras, but some results may be immediately extended to topological algebras as well. In~\cite{ThatteBhatt} a wide class of topological algebras for which topological invertibility collapses to invertibility is provided. 	Therefore, we may ask 
	
	\begin{question}
		Is there a wider class of unital topological algebras where approximate invertibility coincides with invertibility? Provide a characterization of such algebras.
	\end{question}
	
	Recall that $x\in\A$ is approximately left and right invertible in $\A$, if there exist nets $(l_i)_{i\in I}$ and $(r_j)_{j\in J}$ such that the nets $(l_ix)_{i\in I}$ and $(xr_j)_{j\in J}$ are approximate identities in $\A$. Denote by $\AppInv_0(\A)$ the set of all true bilateral approximately invertible elements from $\A$, i.e., $x\in\AppInv_0(\A)$ if and only if there exists a net $(t_k)_{k\in K}$ such that the nets $(t_k x)_{k\in K}$ and $(x t_k)_{k\in K}$ are approximate identities in $\A$. Clearly, for each abelian topological algebra $\A$ it holds
	$\AppInv(\A) = \AppInv_0(\A)$. 
	
	\begin{question}
		Does there exist a non-abelian non-unital topological algebra $\A$ such that $\AppInv(\A) = \AppInv_0(\A)$? \end{question}
	
	Recent Abel results on sets of topologically quasi-invertible sets, see~\cite{A-ZR}, motivate further research on the structure of the sets of $\AppInv_{\ell}(\A)$ and  $\AppInv_r(\A)$ in $\A$. For instance, for which algebras $\A$ are these sets $G_\delta$-sets? This stimulates many questions regarding properties of an extension of the topological spectrum of elements, spectral mapping property, spectral radius, etc. 
	
	It would be interesting to generalize the concept of the spectrum to algebras with approximate identities. This generalization leads to the several questions.
	
	\begin{question}
		Given an element $a\in\A$, a complex number $\lambda$ and an approximate identity $(e_j)_{j\in J}$ in $\A$ we could consider the following condition:
		there exists a net $(r_j)_{j\in J}$ such that $((\lambda e_j - a) r_j)_{j\in J}$ is an approximate identity in $\A$.
		Does this condition depend on the selection of $(e_j)_{j\in J}$?
	\end{question}
	
	Since the approximate invertibility turns out to be related with the density of principal ideals, there may be some connections with tauberian theorems worth of further investigations
	(see recent papers \cite{fe88,fe15,fegu20,pivi19,vipira11}).

	\textbf{Acknowledgements.}
	The first named author wishes to thank the Universidad de Caldas for financial support and hospitality.
	This is a part of the first author’s project ``Elementos aproximadamente invertibles en C*-\'algebras y sus aplicaciones en teor\'ia de operadores''. Third author acknowledges the support of the Slovak Research and Development Agency under the contract No. DS-2016-0028 and APVV-16-0337. The forth author is grateful to the CONACYT (Mexico) project ``Ciencia de Frontera''  FORDECYT-PRONACES/61517/2020.

	\medskip\noindent
	Kevin Esmeral,\\
	\url{https://orcid.org/0000-0003-1147-4730},\\
	e-mail: kevin.esmeral@ucaldas.edu.co,\\
	Department of Mathematics, Universidad de Caldas, \\
	C\'odigo Postal 170004, 
	Manizales, COLOMBIA.
	
	\medskip\noindent
	Hans G. Feichtinger,\\
	\url{https://orcid.org/0000-0002-9927-0742},\\
	email: hans.georg.feichtinger@univie.ac.at,\\
	Faculty of Mathematics, University of Vienna, \\ Oskar-Morgenstern-Platz 1, Wien, 1090, AUSTRIA.
	
	\medskip\noindent
	Ondrej Hutn\'{i}k,\\
	\url{https://orcid.org/0000-0003-1189-9667},\\
	email: ondrej.hutnik@upjs.sk,\\
	Institute of Mathematics, Pavol Jozef \v{S}af\'{a}rik University in Ko\v{s}ice, \\ 
	Jesenn\'{a} 5, 040~01 Ko\v{s}ice, SLOVAKIA.
	
	\medskip\noindent
	Egor A. Maximenko,\\
	\url{https://orcid.org/0000-0002-1497-4338},\\
	e-mail: egormaximenko@gmail.com,\\
	Instituto Polit\'{e}cnico Nacional,\\
	Escuela Superior de F\'{i}sica y Matem\'{a}aticas,\\
	Apartado Postal 07730,
	Ciudad de M\'{e}xico, MEXICO.
	
\end{document}